\documentclass[11pt, notitlepage]{article}
\usepackage{amssymb,amsmath,comment}
\usepackage{epsfig,enumerate,amsmath,amsfonts,amssymb}
\usepackage{indentfirst}
\usepackage{pstricks}
\usepackage{setspace}
\usepackage{mathtools}
\usepackage{latexsym}
\usepackage{bm}
\usepackage{amsthm}
\usepackage{amsmath}
\usepackage{parskip}
\usepackage{fixmath}
\catcode`\@=11 \@addtoreset{equation}{section}
\def\thesection{\arabic{section}}

\def\theequation{\thesection.\arabic{equation}}
\catcode`\@=12
\usepackage{colortbl}
\usepackage{a4wide}

\newcommand{\ds} {\displaystyle}
\newcommand{\e}{\varepsilon}
\newcommand{\pa} {\partial}
\newcommand{\al} {\alpha}
\newcommand{\ba} {\beta}

\newcommand{\ga} {\gamma}
\newcommand{\Om} {\Omega}
\newcommand{\ra} {\rightarrow}
\newcommand{\rp} {\rightharpoonup}
\newcommand{\De} {\Delta}
\newcommand{\la} {\lambda}
\newcommand{\La} {\Lambda}
\newcommand{\noi} {\noindent}
\newcommand{\na} {\nabla}

\newcommand{\mc} {\mathcal}

\newcommand{\ld} {\langle}
\newcommand{\rd} {\rangle}

\newcommand{\vertiii}[1]{{\left\vert\kern-0.25ex\left\vert\kern-0.25ex\left\vert #1 
		\right\vert\kern-0.25ex\right\vert\kern-0.25ex\right\vert}}

\usepackage[all]{xy}
\catcode`\@=11
\def\theequation{\@arabic{\c@section}.\@arabic{\c@equation}}
\catcode`\@=12

\def\R{{I\!\!R}}

\def\QED{\hfill {$\square$}\goodbreak \medskip}

\def\R{\mathbb{R}^N}

\newtheorem{Theorem}{Theorem}[section]
\newtheorem{Lemma}[Theorem]{Lemma}
\newtheorem{Proposition}[Theorem]{Proposition}

\newtheorem{Remark}[Theorem]{Remark}
\newtheorem{Definition}[Theorem]{Definition}

\begin{document}
	
	\title
	{Singular doubly nonlocal elliptic problems with Choquard type critical growth nonlinearities }
		\author{  Jacques Giacomoni$^{\,1}$ \footnote{e-mail: {\tt jacques.giacomoni@univ-pau.fr}}, \ Divya Goel$^{\,2}$\footnote{e-mail: {\tt divyagoel2511@gmail.com}},  \
		and \  K. Sreenadh$^{\,2}$\footnote{
			e-mail: {\tt sreenadh@maths.iitd.ac.in}} \\
		\\ $^1\,${\small Universit\'e  de Pau et des Pays de l'Adour, LMAP (UMR E2S-UPPA CNRS 5142) }\\ {\small Bat. IPRA, Avenue de l'Universit\'e F-64013 Pau, France}\\  
		$^2\,${\small Department of Mathematics, Indian Institute of Technology Delhi,}\\
		{\small	Hauz Khaz, New Delhi-110016, India } }

	\date{}
	
	\maketitle

		\begin{abstract}
		 The theory  of elliptic equations involving singular nonlinearities is well studied topic but the  interaction of singular type nonlinearity with  nonlocal nonlinearity in elliptic problems has not been investigated so far. In this article, we  study the very singular and doubly nonlocal singular problem $(P_\la)$(See below). Firstly, we establish a very weak comparison principle and the optimal Sobolev regularity. Next using the critical point theory  of non-smooth analysis and the geometry of the energy functional, we establish the global multiplicity of positive weak solutions. 
		\medskip
		
		\noi \textbf{Key words}: Choquard equation, Fractional Laplacian, Singular Nonlinearity, Non smooth analysis, Regularity. 
		
		\medskip
		
		\noi \textit{2010 Mathematics Subject Classification:   49J35, 35A15, 35S15, 46E35, 49J52}

	\end{abstract}
\newpage
\section{Introduction}
The purpose of the article is to discuss the existence and multiplicity  of  weak solutions to the following singular problem: 
\begin{equation*}
(P_\la)\;
\left\{\begin{array}{rllll}
(-\De)^s u
&=u^{-q}+ \la \left(\ds \int_{\Om}\frac{|u|^{2^*_\mu}(y)}{|x-y|^{\mu}}dy\right) |u|^{2^*_\mu-2}u, \quad u
>0  \; \text{in}\;
\Om,\\
u&=0 \; \text{ in } \R \setminus \Om,
\end{array}
\right.
\end{equation*}
for all $q>0, N\ge 2s, s \in (0, 1), 2_{\mu}^{*}=\frac{2N-\mu}{N-2s}$ and  $\Om$ is a bounded domain in $\mathbb R^N$ with smooth boundary.  Here the operator  $(-\De)^s$ is the fractional Laplacian  defined as 
\begin{align*}
(-\De)^s u(x)= - \text{P.V. } \int_{ \R} \frac{u(x)-u(y)}{|x-y|^{N+2s}} dy
\end{align*}
where P.V denotes the Cauchy principal value.\\
The problems involving  singular nonlinearity have a  
very long history. In the  pioneering work \cite{crandal}, Crandall, Rabinowitz and Tartar \cite{crandal} proved the existence of a  solution  of  classical elliptic PDE with singular nonlinearity using the approximation arguments. Later many researchers studied the problems involving singular nonlinearity. Haitao \cite{hai} studied the following problem 
\begin{align}\label{si96}
-\De u = au^{-q}+b u^{\frac{N+2}{N-2}}, u>0\text{   in   } \Om, \quad u  =0 \text{   on   } \pa \Om
\end{align}
where $\Om\subset \R(N\geq 3)$ is a smooth bounded domain. If  $a=\la$ and $b=1$, and  $q\in (0,1)$, authors proved a global multiplicity result. While in \cite{adi1,dhanya1}, researchers improvised the results of \cite{hai} and  proved the global multiplicity result for  $q \in (0,3)$.
In \cite{hirano1}, Hirano, Saccon, and Shioji studied the problem \eqref{si96} with $a=\la$ and $b=1$, and  $q\in (0,1)$.  Using the  well known splitting Nehari manifold method, authors proved the multiplicity of solutions for small $\la$.    While in \cite{hirano}, authors  studied the problem for all $q>0$, $a=1~b=\la$,  and established a global multiplicity result  using the nonsmooth analysis. For more details on singular type problems,  we refer to \cite{coc,ghergu,GST2,GST, hai,hern} and references therein.  

The study of nonlinear elliptic problems with critical terms motivated by Hardy-Littlewood-Sobolev inequality started long back and  attracted lot of researchers due to its wide applications. Indeed,  it was originated in the framework of various physical models. One of the first applications was found in   H. Fr\"ohlich and S. Pekar  model of   the polaron, where free electrons in an ionic lattice interact with photons associated to the deformations of the lattice or with the polarization that it creates on the medium \cite{fro1,fro2}.  In the modeling of one component plasma, Ph. Choquard  gave the model which involves  Choquard equation \cite{ehleib}. Later on such nonlinear problems are called Choquard equations and many researchers studied these type of problems to understand the existence, uniqueness, radial symmetry and regularity of the solutions \cite{moroz1,moroz2,moroz5} and references therein. 
Pertaining the Choquard type critical exponent problems on bounded domains, Gao and Yang \cite{yang}  studied the Brezis-Nirenberg type existence and nonexistence results with Choquard critical nonlinearities. In \cite{cassani}, \cite{seok} and\cite{tang} are studied a Brezis-Nirenberg type problem with uppercritical growth, concentration profiles of ground states and existence of semiclassical states respectively.

Nonlocal problems involving fractional Laplacian  challenged a lot of researchers due to the large spectrum of applications.  Consider the   following problem
\begin{align}\label{si66}
(-\De)^s u = f(x,u) \text{ in } \Om ,\quad u=0 \text{ in } \R\setminus \Om
\end{align}
where $f$ is a Carath\'eodory function.   The questions of existence, multiplicity and regularity  of solutions  to problem \eqref{si66} have been extensively studied in \cite{servadei, valdi} and references therein. 
Concerning the   existence and multiplicity of solutions to doubly nonlocal problems,  a lot of works have been done. For a detailed state of art, one can refer \cite{chen,avenia,ts} and references therein.

On the other hand, Barrios et al.  \cite{barrios} started the work on nonlocal equations with singular nonlinearity. Precisely, \cite{barrios} deal with the existence of solutions to  the following problem 
\begin{align}\label{si67}
(-\De)^s u = \la \left(\frac{a(x)}{u^{r}}+ f(x,u) \right) \text{ in } \Om,\quad  u=0 \text{ in } \R\setminus \Om
\end{align}
where $\Om$ is a  bounded domain with smooth boundary, $N>2s,\; 0<s<1,\; r,\la>0,\; f(x,u)\sim u^p,\; 1<p<2^*_s -1$.  In the spirit of \cite{crandal}, here authors first prove the existence of solutions $u_n$ to the equation with singular term $1/u^r$ replaced by $1/(u+1/n)^{r}$ and use the uniform estimates on the sequence $\{u_n\}$ to finally prove the existence of a solution to \eqref{si67}.
Furthermore, authors prove some Sobolev regularity, in particular for $r>1$ that  $u^{\frac{r+1}{2}} \in X_0$.  
In case of $s=1$, optimal Sobolev regularity was established in \cite{bocardo} and \cite{gia} for semilinear and quasilinear elliptic type problems respectively.  But  in case of $0<s<1$, the  question  of optimal Sobolev regularity   still remained an open question.  The regularity issue  is of  independent interest.  In the recent times, Adimurthi, Giacomoni and Santra \cite{adi} studied the problem \eqref{si67} with $f=0$ and  complement the results of \cite{barrios}. In particular they obtained the boundary behaviour and H\"older regularity of the classical solution. Then exploiting this asymptotic behavior, authors obtained multiplicity of classical solutions by global bifurcation method in the framework of weighted spaces for  \eqref{si67} with subcritical $f$. 

 Regarding the critical case, Giacomoni, Mukherjee and Sreenadh \cite{GMS} studied the   problem  \eqref{si67} with $a=1/\la, \; r>0,$ and $f(x,u)\sim u^{2^{*}_{s}-1}$.   Here authors extended the techniques of \cite{hirano} in fractional framework and   proved the existence and multiplicity of solutions  in $C^\al_{\text{loc}}(\Om) \cap L^\infty(\Om)$ for some $\al>0$. Recently, authors \cite{gts1} proved the  global multiplicity result  for \eqref{si67} with $a=1/\la, p=2^*_s-1 $ and $r(2s-1)<(1+2s)$ for energy solutions. 
  Concerning the doubly nonlocal problem with singular operators, in \cite{newpaper}, we studied the regularity results for the problems of th type $(P_\la)$ with $0<q<1$. 
  But the questions of existence, multiplicity of solutions to the problem $(P_\la)$ was a  completely open problem even when $s=1$. Also  the question of (H\"older, Sobolev) regularity of solutions for $q\geq 1$  still remained as an open problem. 

  Motivated by the above discussion, in this article we answer the open problems stated above with an unified approach. More precisely, we  consider a more general  definition of  weak solutions  as compared to definition of \eqref{si67} in \cite{adi}. By establishing a new comparison principle (see Lemma \ref{lemsi17}) we prove that any very weak solution is actually a classical solution. This  is a significant extension of the  regularity results obtained in \cite{adi}. The question of optimal Sobolev regularity is  also answered in our article (see Lemma \ref{lemsi16} and Theorem \ref{thmsi5}). To prove Lemma \ref{lemsi16}, we  exploit suitably the boundary behavior of the  weak solution of problem \eqref{si5}  and the Hardy's inequality. 
The crucial comparison principle  in Lemma \ref{lemsi17} is obtained with a careful use of suitable testing functions  to tackle the $H^s_{loc}$ phenomena. In case of $s=1$, this result was established in \cite{canino} with a slightly different approach. We first prove the $L^\infty(\Om)$ estimate for solutions of  $(P_\la)$ by establishing the relation between  the solutions of $(P_\la)$ and $(\widetilde{P_{\la}})$ (See Section 2). 
The techniques used here can be applied in a more general context and are of independent interest. Next, using the results of  \cite{adi} and Lemma \ref{lemsi17} we prove the  asymptotic behavior and  optimal H\"older and Sobolev regularity of weak solutions. 

In this paper we have given a consolidated approach  to prove  the global multiplicity result for the problem $(P_\la)$ exploiting convex properties of the singular nonlinearity and the geometry of the energy functional. 
To the best of our knowledge there is no previous contribution which deals the Choquard problem with singular nonlinearity.  Further, the results proved in this article are new and novel even in case of $s=1$ where the approach  can be closely adapted.

For simplicity of illustration, we set some notations. We denote $\|u\|_{L^p(\Om)} $ by  $|u|_p$,   $\|u\|_{X_0}$ by $\|u\|$, $[u]_{H^s(A)}= \int_{A}\int_{A}\frac{(u(x)-u(y))^2}{|x-y|^{N+2s}}~dxdy$. The positive constants $C,c_1,c_2\cdots $ values change case by case.

 Turning to the paper organization: In Section 2, we define the function spaces,  give  some preliminaries  of nonsmooth analysis  and further state the  main results of the article. In Section 3,  we establish a very weak comparison principle. In Section 4, we established the regularity of solutions to $(P_\la)$. In sections 5 and 6 , we prove the existence of first and second  solution to $(P_\la)$. 
 
\section{Preliminary results and statement of main results}
We recall the Hardy-Littlewood-Sobolev Inequality which is foundational in study of Choquard problems of the type $(P_\la)$
\begin{Proposition} \cite{leib}
	Let $t,r>1$ and $0<\mu <N$ with $1/t+\mu/N+1/r=2$, $f\in L^t(\mathbb{R}^N)$ and $h\in L^r(\mathbb{R}^N)$. There exists a sharp constant $C(t,r,\mu,N)$ independent of $f,h$, such that 
	\begin{equation*}\label{co9}
	\int_{\mathbb{R}^N}\int_{\mathbb{R}^N}\frac{f(x)h(y)}{|x-y|^{\mu}}~dydx \leq C(t,r,\mu,N) |f|_{t}|h|_{r}.
	\end{equation*}
\end{Proposition}
Consider the space 
\begin{align*}\label{si68}
X_0:=  \{u \in H^{s}(\R): u=0  \text{ a.e in  } \R \setminus \Om \},
\end{align*}
equipped with the inner product 
\begin{align*}
\ld u,v \rd = \iint_{Q} \frac{(u(x)-u(y))(v(x)-v(y))}{|x-y|^{N+2s}}~dxdy.
\end{align*} 
 where $Q= \mathbb{R}^{2N} \setminus (\Om^c\times \Om^c)$. From the embedding results (\cite{servadei}), the space $X_0$ is continuously embedded  into $L^r(\R)$ with $ r\in [1,2^*_s]$ where $2^*_s= \frac{2N}{N-2s}$. The embedding is compact if and only if $r < 2^{*}_{s}$.
The best constant $S$ of the classical Sobolev embedding is defined  
\begin{align*}
S= \inf_{u \in X_0\setminus \{ 0\}}   \frac{\int_{\R}\int_{\R} \frac{|u(x)-u(y)|^2}{|x-y|^{N+2s}}~dxdy }{\left( \int_{ \Om} |u|^{2^*_s}\right)^{2/2^*_s}}. 
\end{align*}
Consequently, we define
\begin{align*}
S_{H,L}= \inf_{u \in X_0\setminus \{ 0\}}   \frac{\int_{\R}\int_{\R} \frac{|u(x)-u(y)|^2}{|x-y|^{N+2s}}~dxdy }{\left( \int_{\R}\int_{\R} \frac{|u|^{2^*_\mu}(x)|u|^{2^*_\mu}(y)}{|x-y|^\mu}  ~dxdy \right)^{1/2^*_\mu}}. 
\end{align*}
\begin{Lemma}\label{lulem13}
	\cite{ts}
	The constant $S_{H,L}$ is achieved if  and only if 
	\begin{align*}
	u=C\left(\frac{b}{b^2+|x-a|^2}\right)^{\frac{N-2s}{2}}
	\end{align*} 
	where $C>0$ is a fixed constant , $a\in \mathbb{R}^N$ and $b\in (0,\infty)$ are parameters. Moreover,
	\begin{align*}
	S=	S_{H,L} \left(C(N,\mu)\right)^{\frac{N-2s}{2N -\mu}}.
	\end{align*}
\end{Lemma}

\begin{Definition} 
	For $\phi \in C_0(\overline{\Om})$ with $\phi >0$ in $\Om$, the set $C_\phi(\Om)$ is defined as 
	\begin{align*}
	C_\phi(\Om)= \{ u \in C_0(\overline{\Om})\; :\; \text{there exists } c \geq 0 \text{ such that } |u(x)|\leq c\phi(x), \text{ for all } x \in \Om   \},
	\end{align*}
	endowed with the natural norm $\bigg\|\ds \frac{u}{\phi}\bigg\| _{L^{\infty}(\Om)}$.
\end{Definition}
\begin{Definition}
	The positive cone of $C_\phi(\Om)$ is the open convex subset of $C_\phi(\Om)$ defined as 
	\begin{align*}
	C_\phi^+(\Om)= \left\{ u \in C_\phi(\Om)\; :\; \inf_{x \in \Om} \frac{u(x)}{\phi(x)}>0 \right\}.
	\end{align*}
\end{Definition}


The barrier function  $\phi_q$ is defined as follows: 
\begin{equation*}\label{si3}
\begin{aligned}
\phi_q=  \left\{
\begin{array}{ll}
\phi_1  &  \text{ if }0<q <1, \\
\phi_1\left(\log\left( \frac{A}{\phi_1}\right) \right)^{\frac{1}{2}} & \text{ if } q=1, \\
\phi_1^{\frac{2}{q+1}} & \text{ if } q>1,\\
\end{array} 
\right.
\end{aligned}
\end{equation*}
where $\phi_1$  is  the  normalized ($\|\phi_1\|_{L^\infty(\Om)}=1$)  eigenfunction corresponding to the smallest eigenvalue of $(-\De)^{s}$ on $X_0$ and $A> \text{diam}(\Om)$. We recall that $\phi_1 \in C^{s}(\R)$  and $ \phi_1 
\in C_{d^{s}}^+(\Om)$ (See Proposition 1.1 and Theorem 1.2 of \cite{RS}).

Before giving the definition of weak solution to $(P_\la)$ we  discuss the  solution of 
 the following purely singular problem
\begin{equation}\label{si5}
(-\De)^s u
=u^{-q},\; u>0   \; \text{in}\;
\Om,
u=0 \; \text{ in } \R \setminus \Om.
\end{equation}
From 	\cite{adi} we know the following result. 
\begin{Proposition}\label{propsi1}
Let $q>0$. Then there exists $ \overline{u} \in  L^\infty(\Om) \cap C^+_{\phi_q} $ classical solution to  \eqref{si5}. Moreover,  $\overline{u}$ has the following properties:
\begin{enumerate}
	\item [(i)] $\overline{u} \in X_0$ if  and only if   $q(2s-1) < (2s+1)$ and in this case we have unique classical solution to \eqref{si5}.
	\item [(ii)] $\overline{u} \in  C^\ga(\R) $ where 
		\begin{equation}\label{si4}
	\begin{aligned}
	\ga=  \left\{
	\begin{array}{ll}
	s  &  \text{ if  }q<1, \\
	s-\e  & \text{ if } q=1, \text{ for all } \e>0 \text{ small enough}, \\
	\frac{1}{q+1} & \text{ if } q>1.\\
	\end{array} 
	\right.
	\end{aligned}
	\end{equation}
\end{enumerate}	
\end{Proposition}
\begin{Remark}\label{remsi1}
	We remark that since $\overline{u} \in L^\infty(\Om) \cap C^+_{\phi_q}  \cap C^\ga(\R)$. So we can achieve the interior $C^\infty$ regularity. That is for any compact set $\Om^\prime \subset \Om$ we have $\overline{u} \in C^\infty(\Om^\prime)$. From this one can easily prove the fact that $\overline{u} \in H^s_{\text{loc}}(\Om)$. 
\end{Remark}
\begin{Lemma}\label{lemsi16}
(a)	If  $q(2s-1) \geq  (2s+1)$  then $\overline{u}^\ga \in X_0$ if and only if $\ga>\frac{(2s-1)(q+1)}{4s}$. Moreover the lower bound on $\ga$ is optimal in the sense that $u^\ga \not \in X_0$ if $\ga\leq \frac{(2s-1)(q+1)}{4s}$. \\
(b)  $(\overline{u}-\e)^+ \in X_0$ for all $\e>0$. 
\end{Lemma}
\begin{proof}
	(a) Let $\xi(x)= x^\ga,\; \ga >1$. Observe that $\xi$ is convex and differentiable function on $\mathbb    R^+$. Hence using this and the fact that $\phi_ 1 \in  C_{d^{s}}^+(\Om)$ and $\overline{u} \in C^+_{\phi_q}$, we deduce that 
	\begin{align*}
	\|\xi(\overline{u})\|^2= \ld (-\De)^s \xi(\overline{u}), \xi(\overline{u})\rd & \leq \int_{\Om} \xi(\overline{u})\xi^\prime(\overline{u}) (-\De)^s \overline{u}~dx\\
	& = \int_{\Om} \ga \overline{u}^{2\ga-q-1}~dx\\
	& \leq C \int_{\Om} d^{\frac{2s(2\ga-q-1)}{q+1}}~dx.
	\end{align*}
	We know that  $d^{\frac{2s(2\ga-q-1)}{q+1}} \in L^1(\Om)$ if  and only if $\ga > \frac{(2s-1)(q+1)}{4s}$. This settles first part of the proof. For the second part, let $\ga\leq \frac{(2s-1)(q+1)}{4s}$ and if possible  let $\overline{u}^\ga \in X_0$. 
	Consider 
	\begin{align*}
	\int_{\Om} \frac{\overline{u}^{2\ga}}{d^{2s}}~dx \leq C\int_{\Om} d^{\frac{4s\ga}{q+1}-2s}~dx = \infty
	\end{align*}
	 It contradicts the fact that $\overline{u}^\ga \in X_0$  and then satisfies the Hardy inequality.\\
	 (b) Let  $A= \{x: \overline{u}(x)>\e \}$ then 
	 \begin{align*}
	 \|(\overline{u}-\e)^+\|^2&  \leq  \int_{A}\int_{A}\frac{|\overline{u}(x)-\overline{u}(y)|^2}{|x-y|^{N+2s}}~dxdy +2 \int_{\R\setminus A} \int_{A}\frac{|\overline{u}(x)-\overline{u}(y)|^2}{|x-y|^{N+2s}}~dxdy\\
	 & \leq \kappa_1 \int_{A}\int_{A}\frac{|\overline{u}^\ga(x)-\overline{u}^\ga(y)|^2}{|x-y|^{N+2s}}~dxdy +2\kappa_2 \int_{\R\setminus A} \int_{A}\frac{|\overline{u}^\ga(x)-\overline{u}^\ga(y)|^2}{|x-y|^{N+2s}}~dxdy \\
	 &	\leq (\kappa_1+\kappa_2) \int_{Q}\frac{|\overline{u}^\ga(x)-\overline{u}^\ga(y)|^2}{|x-y|^{N+2s}}~dxdy < \infty. 
	 \end{align*}
By the mean value theorem and convexity arguments,  one can easily prove the existence of $\kappa_1, \kappa_2$  such  that  $\kappa_1, \kappa_2 \geq \e^{\ga -1}$.   Hence the proof is complete. \QED	
\end{proof}
The  energy functional associated to the probelm $(P_\la)$ is
\begin{align*}
I(u)= \frac12 \int_{Q} \frac{|u(x)-u(y)|^2}{|x-y|^{N+2s}}~dxdy - \frac{1}{1-q} \int_{\Om}|u|^{1-q}~dx - \frac{\la}{22^*_\mu} \iint_{\Om\times \Om}\frac{|u|^{2^*_\mu}|u|^{2^*_\mu}}{|x-y|^{\mu}}~dxdy. 
\end{align*}
Though the functional  $I$  is continuous on $X_0$ when  $0 < q < 1$ but if  $q \geq 1$, the functional $I$ is not  even finite at all points of $X_0$.  
Also $I$  it can be shown that $I$ is not  G\^ateaux differentiable at all points of $X_0$. 
The doubly nonlocal nature of the problem $(P_\la)$ and the lack of regularity of $I$ force to use to introduce a quite general definition of weak solution. The  Lemma \ref{lemsi16} motivates the following  definition of weak solution to the problem $(P_\la)$. 
\begin{Definition}\label{defisi1}
	A function $u \in H^s_{\text{loc}}(\Om)\cap L^{2^*_s}(\Om)$  is said to be a weak solution of $(P_\la)$ if  the following hold: 
	\begin{enumerate}
		\item [(i)] there exists $m_K>0$ such that $ u >m_K$ for any compact set $ K \subset \Om$. 
		\item [(ii)]  For any $ \phi \in C_c^\infty(\Om)$, 
		\begin{align*}
		\langle u, \phi \rangle= \int_{\Om} u^{-q}\phi ~dx+  \la\iint_{\Om\times \Om} \frac{u^{2^*_\mu}(x)u^{2^*_\mu-1}(y)\phi(y)}{|x-y|^\mu}  ~dxdy. 
		\end{align*}
		\item [(iii)] $(u-\e)^{+} \in X_0$ for all $\e>0$. 
	\end{enumerate}
\end{Definition}

\begin{Lemma}\label{lemsi20}
	Let $u$ be a weak  solution  to $(P_\la)$ as it is defined in Definition \ref{defisi1}. Then for  all compactly supported $0\leq v \in X_0\cap L^\infty(\Om)$, we have 
	\begin{align}\label{si83}
	\ld u, v \rd  -  \int_{ \Om} u^{-q}v~dx -\iint_{\Om\times \Om}\frac{u^{2^*_\mu}u^{2^*_\mu-1}v}{|x-y|^{\mu}}~dxdy=0. 
	\end{align}
\end{Lemma}
\begin{proof}
	 Let  $0\leq v \in X_0\cap L^\infty(\Om)$   be  compactly supported function. Then   there exists a sequence $v_n \in C_c^\infty(\Om)$  such that $v_n \geq 0, ~ K:= \cup_n \text{supp } v_n$ is contained in   compact set of $\Om ,~\{ |v_n|_\infty\} $ is bounded sequence and $v_n \ra v$ strongly in $X_0$.  Since $u$ is a weak solution to $(P_\la)$,  we have
	\begin{align}\label{si84}
	\ld u, v_n \rd - \int_{ \Om} u^{-q}v_n~dx -\iint_{\Om\times \Om}\frac{u^{2^*_\mu}u^{2^*_\mu-1}v_n}{|x-y|^{\mu}}~dxdy=0. 
	\end{align}
	Consider 
	\begin{align*}
	\int_{Q} \frac{(u(x)-u(y))(v_n(x)-v_n(y))}{|x-y|^{N+2s}}~dxdy & = 	\int_{K}\int_{K} \frac{(u(x)-u(y))(v_n(x)-v_n(y))}{|x-y|^{N+2s}}~dxdy\\
	&  +2 \int_{\Om\setminus K}\int_{K} \frac{(u(x)-u(y))v_n(x)}{|x-y|^{N+2s}}~dxdy\\
	&  +  2 \int_{\R\setminus \Om}\int_{\Om} \frac{u(x)v_n(x)}{|x-y|^{N+2s}}~dxdy\\
	&:= I+II+III.
	\end{align*}
		Now using the fact that $u \in H^s_{\text{loc}}(\Om)$ and the strong convergence of $v_n$, as $n \ra \infty$,  we obtain 
	\begin{align*}
I\ra \int_{K}\int_{K} \frac{(u(x)-u(y))(v(x)-v(y))}{|x-y|^{N+2s}}~dxdy.
	\end{align*}
	Next taking into account $u \in L^{2^*_s}(\Om)$ and the fact that $ v_n $ is  uniformly bounded in $L^\infty(\Om)$, by dominated convergence theorem, we deduce that 
	\begin{align*}
II\ra 	2 \int_{\Om\setminus K}\int_{K} \frac{(u(x)-u(y))v(x)}{|x-y|^{N+2s}}~dxdy.
	\end{align*}
	Similarly, $III\ra 2 \int_{\R\setminus \Om}\int_{\Om} \frac{u(x)v(x)}{|x-y|^{N+2s}}~dxdy$. Hence $\ld u , v_n \rd \ra \ld u , v \rd $ as $n \ra \infty$. Trivially $\int_{ \Om} u^{-q}v_n~dx \ra \int_{ \Om} u^{-q}v~dx$ as $n \ra \infty$. 
	Also  using the strong convergence of sequence $v_n$ and the fact that $u \in L^{2^*_s}(\Om)$, we infer that 
	\begin{align*}
	\iint_{\Om\times \Om}\frac{u^{2^*_\mu}u^{2^*_\mu-1}v_n}{|x-y|^{\mu}}~dxdy\ra \iint_{\Om\times \Om}\frac{u^{2^*_\mu}u^{2^*_\mu-1}v}{|x-y|^{\mu}}~dxdy.
	\end{align*}
	It implies that passing the limit as $n \ra \infty$ in \eqref{si84}, we have \eqref{si83} for  all compactly supported $0\leq v \in X_0\cap L^\infty(\Om)$. \QED
\end{proof}

In the direction of existence of solution to $(P_\la)$, we translate the problem $(P_\la)$ by the solution $\overline{u}$ of  problem \eqref{si5}.  
Consider the translated problem 
\begin{equation*}
(\widetilde{P_\la})\;
\left\{\begin{array}{rllll}
(-\De)^s u +\overline{u}^{-q}-(u+\overline{u})^{-q}
&= \la \left(\ds \int_{\Om}\frac{(u+\overline{u})^{2^*_\mu}(y)}{|x-y|^{\mu}}dy\right) (u+\overline{u})^{2^*_\mu-1},\; u>0  \; \text{in}\;
\Om,\\
u&=0 \; \text{ in } \R \setminus \Om.
\end{array}
\right.
\end{equation*}
Observe that $u+\overline{u} $ is a solution to $(P_\la)$ if and only if  $u\in X_0$ is a solution  to $(\widetilde{P_\la})$. Define the function $g: \Om \times \mathbb{R}\ra \mathbb{R}\cap \{ \infty\}$ as 
\begin{align*}
g(x,s)=  \left\{
\begin{array}{ll}
\overline{u}^{-q}-(s+\overline{u})^{-q} &  \text{ if }  s+\overline{u}>0 , \\
-\infty  & \text{ otherwise}  
\end{array} 
\right.
\end{align*}
and $G(x, y) = \int_{0}^y g(x,\tau)~d\tau$  for $(x,y ) \in \Om\times \mathbb{R}$. 
\begin{Definition}
	A function $u \in X$ is a subsolution (resp. supersolution) to $(\widetilde{P_\la})$ if the following holds
	\begin{enumerate}
		\item [(i)] $u^+ \in X_0$(resp. $u^- \in X_0$);
		\item  [(ii)] $g(\cdot, u) \in L^1_{\text{loc}}(\Om)$;
		\item [(iii)] For all $\phi \in C_c^\infty(\Om),\; \phi \geq 0$, we have 
		\begin{align*}
		\ld u, \phi \rd  + \int_{ \Om} g(x,u)\phi~dx - \iint_{\Om\times \Om}\frac{(u+\overline{u})^{2^*_\mu}(u+\overline{u})^{2^*_\mu-1}\phi}{|x-y|^{\mu}}~dxdy\leq 0 \; (\text{resp.}\geq 0  ). 
		\end{align*}
	\end{enumerate}
\end{Definition}

\begin{Definition}
	A function $u \in X_0$ is a  weak solution to $(\widetilde{P_\la})$ if it is both sub and supersolution and $
	u \geq 0 $ in $\Om$. 
\end{Definition}
%

\begin{Lemma}\label{lemsi18}
	Let $u \in X_0$ be a weak  solution  to $(\widetilde{P_{\la}})$. Then for any $v \in X_0$, we have 
	\begin{align}\label{si81}
		\ld u, v \rd  + \int_{ \Om} g(x,u)v~dx - \iint_{\Om\times \Om}\frac{(u+\overline{u})^{2^*_\mu}(u+\overline{u})^{2^*_\mu-1}v}{|x-y|^{\mu}}~dxdy=0. 
	\end{align}
\end{Lemma}
\begin{proof}
 Let $0\leq v \in X_0$ then  by \cite[Lemma 3.1]{GMS}, there exists an increasing  sequence $\{v_n  \} \in X_0$ such that  $v_n$ has a compact support and $v_n \ra v $ strongly in $X_0$. For each $n$, there exists a sequence $\phi_n^k \in C_c^\infty(\Om)$  such that $\phi_n^k \geq 0, ~ \cup_k \text{supp} \phi_n^k$ is contained in   compact set of $\Om,~\{ |\phi_n^k|_\infty\} $ is bounded sequence and $\phi_n^k \ra v_n$ strongly in $X_0$. Since $u$ is a weak solution to $(\widetilde{P_{\la}})$,  we have
		\begin{align*}
	\ld u, \phi_n^k \rd  + \int_{ \Om} g(x,u)\phi_n^k~dx - \iint_{\Om\times \Om}\frac{(u+\overline{u})^{2^*_\mu}(u+\overline{u})^{2^*_\mu-1}\phi_n^k}{|x-y|^{\mu}}~dxdy=0. 
	\end{align*}
	Using the fact that $\phi_n^k \ra v_n$ strongly in $X_0$ as $n \ra \infty$, we deduce that  
	\begin{align*}
	\ld u, v_n \rd  + \int_{ \Om} g(x,u)v_n~dx - \iint_{\Om\times \Om}\frac{(u+\overline{u})^{2^*_\mu}(u+\overline{u})^{2^*_\mu-1}v_n}{|x-y|^{\mu}}~dxdy=0. 
	\end{align*}
	Now by using the dominated convergence theorem and the strong convergence of the sequence $v_n$ in $X_0$, we get $g(x,u)v \in L^1(\Om)$ and  we have \eqref{si81} for any $0\leq v \in X_0$. For any $v \in X_0$, $v= v^+-v^-$. Employing the above procedure for $v^+$ and $v^-$ separately, we obtain the desired result. Hence the proof.	 \QED
\end{proof}
With this functional framework we record now the statement of our main Theorems. 
\begin{Theorem}\label{thmsi1}
	 Let $\mu \leq \min\{ 4s,N\}$. There exists a $\La >0$ such that
	\begin{enumerate}
		\item For every $\lambda \in (0, \La)$ the problem $(P_\lambda)$ admits two solutions in $ C_{\phi_q}^+(\Om)\cap L^\infty(\Om)$. 
		\item For $\lambda = \La$ there exists a solution in $ C_{\phi_q}^+(\Om)\cap L^\infty(\Om)$.
		\item For $\lambda > \La$, there exists no solution.
	\end{enumerate}
Moreover, solution belongs to $ X_0$ if and only  if $q(2s-1)<(2s+1)$.
\end{Theorem} 
Concerning the H\"older and Sobolev regularity of solutions we have the following Theorem. 
\begin{Theorem}\label{thmsi5}
Let $\mu \leq \min\{ 4s,N\}$. Let $q>0,\; \la \in (0,\La)$. Then
any weak solution in the sense of  Definition \ref{defisi1} is classical and belongs to $C^\ga(\R)$ where $\ga$ is defined \eqref{si4}. Furthermore any weak solution satisfies the statements of Lemma \ref{lemsi16}.
\end{Theorem}
\begin{Remark}
		We point out that regularity results contained in Theorem \ref{thmsi5} are much stronger
	as compared to those obtained in \cite{adi, gts1} where regularity of continuous solutions are only investigated. 
\end{Remark}

\subsection{Notions of nonsmooth Analysis}
 In this subsection we record some basic definitions,  observations and linking theorem to nonsmooth functionals. We remark that in case of $q(2s-1)<(2s+1)$, one can adapt the  variational techniques of the \cite{hai,gts1} to prove the  global multiplicity result as in Theorem \ref{thmsi1} but to incorporate the case of $q$ large we adopt the following notions of non-smooth analysis. 
\begin{Definition}
	Let $H$ be a Hilbert space and $J : H \ra (-\infty, \infty]$ be a proper (i.e. $J \not \equiv \infty$)  lower semicontinuous
	functional.
	\begin{enumerate}
		\item [(i)]Let $D(J) = \{u\in  H\; : \; J(u) < \infty\}$ be the domain of $J$. For every $u  \in D(J)$, we define the Fr\'echet sub-differential	of $J$ at $u$ as the set
		\begin{align*}
		\pa^-J(u) = \left\{ z \in H \; :\;  \varliminf_{v \ra u }\frac{ J(v)-J(u)- \ld z,v-u\rd}{\|v-u\|_{H}}\geq 0   \right\}. 
		\end{align*}
		\item [(ii)] For each  $u \in  H$, we define
		\begin{align*}
	\vertiii{\pa^-J(u)} = \left\{
		\begin{array}{ll}
		\min \{  \|z\|_H\; :\; z \in \pa^- J(u) \}  &  \text{ if } \pa^- J(u) \not = \emptyset, \\
		\infty & \text{ if } \text{ if } \pa^- J(u)  = \emptyset. \\
		\end{array} 
		\right. 
		\end{align*}
		We know that $ \pa^-J(u)$  is a closed convex set which may be empty. If $u \in  D(J)$ is a local minimizer for $J$, then it
		can be seen that $0 \in \pa^-J(u)$.  
	\end{enumerate}
\end{Definition}

\begin{Remark}
	We remark that if $J_0  :  H \ra (-\infty, \infty]$ be a proper,  lower semicontinuous, convex 
	functional, $J_1 : H \ra \mathbb{R}$  is a $C^1$ functional and $J= J_1 + J_0$, then 
$\pa^{-}J(u) = \na J_1(u) + \pa J_0(u)$  for every $ u \in D(J) = D(J_0)$, where
	$\pa J_0$  denotes the usual subdifferential of the convex functional $J_0$. Thus, $u$ is said to be a critical point of $J$ if
	$ u \in  D(J_0)$  and for every $ v \in  H$, we have
	$ \ld \na J_1(u), v-u\rd +J_0(v) -J_0(u)\geq 0$. 
\end{Remark}
\begin{Definition}
	For a proper, lower semicontinuous functional $J :  H \ra (-\infty, \infty]$, we say that $J$ satisfies Cerami's
	variant of the Palais-Smale condition at level $c$ (in short, $J$ satisfies $(\text{CPS})c$, if any sequence $\{z_n\} \subset  D(J)$
	such that $ J(z_n) \ra c $ and $(1 + z_n) \vertiii{\pa^-J(z_n)} \ra 0$  has a strongly convergent subsequence in $H$.
\end{Definition}
Analogous to the mountain pass theorem, we have the following linking theorem for non-smooth functionals.

\begin{Theorem}\label{thmsi4} \cite{sz}
	Let $H$ be a Hilbert space. Assume $J= J_0+J_1$, where $J_0 : H \ra (-\infty, \infty]$  is
	a proper, lower semicontinuous, convex functional and $J_1 : H \ra \mathbb{R}$ is a $C^1$-functional. Let $B^N , S^{N-1}$ denote
	the closed unit ball and its boundary in $\R$ respectively. Let $\varphi : S^{N-1} \ra D(J)$ be a  continuous function such
	that 
	\begin{align*}
	\Sigma = \{  \psi \in C(B^N, D(J))\; :\; \psi|_{S^{N-1}} = \varphi  \} \not = \emptyset.
	\end{align*}
	Let $A$ be a relatively closed subset of $D(J)$ such that 
	\begin{align*}
	A \cap \varphi(S^{N-1}) = \emptyset, \quad 	A \cap \varphi(B^N) \not =  \emptyset \text{ for all } \psi \in \Sigma \text{ and } \quad \inf J(A)\geq \sup J(\varphi(S^{N-1})). 
	\end{align*}
	Define $c = \inf_{\psi \in \Sigma} \sup_{x \in B^N} J(\psi(x))$. 
	Assume that $c$ is finite and that $J$ satisfies $(\text{CPS})c)$.  Then there exists $u \in  D(J)$  such that $J(u) = c$ and $ 0 \in \pa^-J(u)$. 	Furthermore, if $\inf J(A) = c$, then there exists $u \in A \cap D(J)$ such that $J(u) = c$ and $0 \in \pa^-J (u)$. 
\end{Theorem}

\section{Very weak comparison principle}
Here we  establish a new weak comparison principle that can be applied in the setting of $H^s_{loc}(\Om)$ sub and supersolutions to $(P_\lambda)$ and cover all $q>0$ (whereas \cite[Lemma 2.2]{gts1} $q(2s-1)<2s+1$ is required). 
\begin{Lemma}\label{lemsi17}
	Let $F \in X_0^*$ and let $z,w \in H^s_{\text{loc}}(\Om)$ be such that $z,w  >0$ a.e in $\Om,~ z,w \geq 0 \in \R,~ z^{-q}, w^{-q} \in L^1_{\text{loc}} (\Om), ~ (z-\e)^+ \in X_0$ for all $\e>0,~ z \in L^1(\Om)$ and 
	\begin{align}\label{si71}
	\ld z, \phi \rd \leq \int_{\Om} z^{-q} \phi ~dx + (  F,\phi ) , \quad \ld w,\phi  \rd \geq \int_{\Om} w^{-q} \phi ~dx + ( F,\phi )
	\end{align}
	for all compactly supported  $ \phi \in X_0 \cap L^\infty(\Om) $ with $\phi\geq 0 $. Then $z\leq w$ a.e in $\Om$. 
\end{Lemma}

\begin{proof}
	Let us denote that $\Psi_n : \mathbb{R} \ra \mathbb{R}$ the primitive of the function 
	\begin{align*}
	s \mapsto 
	\left\{
	\begin{array}{ll}
	\max \{  -s^{-q}, -n\}	&  \text{ if }  s>0 , \\
	-n & \text{ if  }  s \leq 0
	\end{array} 
	\right.
	\end{align*}
	such that $\Psi_n (1)=0$. Let us define a proper lower semicontinuous, strictly convex functional $\breve{H}_{0,n} : L^2(\Om) \ra  \mathbb{R}$ given by 
	\begin{align*}
	\breve{H}_{0,n}(u)= 
	\left\{
	\begin{array}{ll}
	\frac12 \|u\|^2 + \int_{\Om}\Psi_n(u)~dx  &  \text{ if }  u \in X_0 , \\
	\infty  & \text{ if } u \in L^2(\Om) \setminus X_0.   
	\end{array} 
	\right.
	\end{align*}
	We define $H_{0,n}: L^2(\Om) \ra  \mathbb{R}$ as 
	\begin{align*}
	H_{0,n}(u) = 	\breve{H}_{0,n}(u) - \min	\breve{H}_{0,n}= 	\breve{H}_{0,n}(u) - 	\breve{H}_{0,n}(u_{0,n})
	\end{align*}
	where $u_{0,n} \in X_0$ is the minimum of $	\breve{H}_{0,n}$. More generally,  for $F \in  X_0^*$ we set:
	\begin{align*}
	\breve{H}_{F,n}(u)= 
	\left\{
	\begin{array}{ll}
	\breve{H}_{0,n}(u) -( F, u- u_{0,n}) &  \text{ if }  u \in X_0 , \\
	\infty  & \text{ if } u \in L^2(\Om) \setminus X_0.   
	\end{array} 
	\right.
	\end{align*}
	Let $\e>0$ and $n > \e^{-q}$ and let $v$ be the minimum of the functional $\breve{H}_{F,n}$ on the convex set $K= \{ \varphi \in X_0: 0\leq \varphi \leq w \text{ a.e  in } \Om  \}$. Then for all $\varphi \in K$ we get 
	\begin{align}\label{si91}
	\ld v, \varphi- v \rd \geq - \int_{\Om}\Psi_n^\prime(v) (\varphi - v )~dx + ( F, \varphi -v ). 
	\end{align}
	Let $0\leq \varphi \in C_c^\infty(\Om), \; t>0$. Define   $\varphi_t := \min \{ v+ t \varphi , w  \} $. Now using the fact that $ w \in H^s_{\text{loc}} (\Om), \; v \in X_0,\;  \varphi 
		\in C^\infty_c(\Om)$, we have $\varphi_t \in  X_0$. Furthermore,  $\varphi_t$ is uniformly bounded in $X_0$ for all $t<1$.
		For the proof let  $A= \text{supp}(\varphi)$. Since  on $\Om \setminus A,~ \varphi_t = v$ and otherwise $v\leq \varphi_t \leq w$,  we deduce that 
		
		\begin{equation}
		\begin{aligned}\label{si89}
		\int_{Q} \frac{(\varphi_t(x)-\varphi_t(y))^2}{|x-y|^{N+2s}}~dxdy&  \leq  [\varphi_t]_{H^s(A)}  +\int_{ \Om \setminus A}\int_{\Om \setminus A}\frac{(v(x)- v(y))^2}{|x-y|^{N+2s}}~dxdy \\
		&~+ 2 \int_{ \Om \setminus A}\int_{A}\frac{ (v(x)- v(y))^2+ 2t  v(y)\varphi(y)+t^2\varphi^2(y)}{|x-y|^{N+2s}}~dxdy
		\\
		&\qquad 	+ 2 \int_{ \R \setminus\Om}\int_{\Om}\frac{(v+t\varphi(y))^2}{|x-y|^{N+2s}}~dxdy <\infty. 
		\end{aligned}
		\end{equation} 
		Employing the fact that for any $g:\R\ra \mathbb{R}$,  $|g^+(x)-g^+(y)|^2\leq |g(x)-g(y)|^2$ for all $x,y \in \R$ coupled with $\varphi_t= v+t\varphi - (v+t\varphi-w)^+$, for all $t<1$, we conclude that 
		\begin{equation}
		\begin{aligned}\label{si90}
		[\varphi_t]_{H^s(A)} & \leq  
		2 \int_{  A}\int_{A}\frac{((v+t\varphi)(x)- (v+t\varphi)(y))^2}{|x-y|^{N+2s}}~dxdy\\
		&  \qquad  + 2
		\int_{  A}\int_{A}\frac{((v+t\varphi-w)(x)- (v+t\varphi-w))^2}{|x-y|^{N+2s}}~dxdy\\
		& \quad \leq   C\left([v]_{H^s(A)}+[\varphi]_{H^s(A)} + [w]_{H^s(A)} \right) . 
		\end{aligned}
		\end{equation}
		From \eqref{si89} and \eqref{si90} we obtain that $\varphi_t$ is uniformly bounded in $X_0$. Take the subsequence (still denoted by $\varphi_t$) such that $\varphi_t\rp v$ weakly in $X_0$  as $t\to 0^+$. Now test \eqref{si91} with $\varphi_t$, we get 
		\begin{align}\label{si72}
		\ld v, \varphi_t- v \rd \geq - \int_{\Om}\Psi_n^\prime(v) (\varphi_t - v )~dx + ( F, \varphi_t -v ). 
		\end{align}
	Using \eqref{si71} and the fact that $w^{-q}\geq -\Psi_n^\prime(w)$,  we infer that $w$ satisfies 
	\begin{align}\label{si73}
	\ld w,\phi  \rd \geq - \int_{\Om}\Psi_n^\prime(w) \phi ~dx+ ( F,\phi )
	\end{align}
	Deploying the fact that  $ \varphi_t \leq w$ coupled with  $\varphi_t- v- t\varphi \leq 0 $ and if  $\varphi_t = w$ then 
	$\varphi_t- v- t\varphi \not = 0$,  we deduce that 
	\begin{equation}\label{si74}
	\begin{aligned}
	& \int_{Q} \frac{(\varphi_t(x)- \varphi_t(y))((\varphi_t-v-t\varphi ) (x)- (\varphi_t-v-t\varphi ) (y))}{|x-y|^{N+2s}}~dxdy\\
	&\leq  \int_{Q} \frac{w(x)(\varphi_t-v-t\varphi ) (x)}{|x-y|^{N+2s}}~dxdy + \int_{Q} \frac{w(y)(\varphi_t-v-t\varphi ) (y)}{|x-y|^{N+2s}}~dxdy \\
	& \quad - \int_{Q} \frac{w(x)(\varphi_t-v-t\varphi ) (y)}{|x-y|^{N+2s}}~dxdy  - \int_{Q} \frac{w(y)(\varphi_t-v-t\varphi ) (x)}{|x-y|^{N+2s}}~dxdy 
	=  \ld w, \varphi_t-v-t\varphi \rd.  
	\end{aligned}
	\end{equation}
	Similarly, $\int_{\Om}(\Psi_n^\prime(\varphi_t)- \Psi_n^\prime(w)) (\varphi_t-v-t\varphi )~dx \leq 0$ and moreover  $\Psi_n^\prime(w) \leq - w^{-q}$. Taking into account \eqref{si71},  \eqref{si72}, \eqref{si73},   \eqref{si74} and above observations,  we deduce that 
	\begin{align*}
	\|\varphi_t-v\|^2 &  - \int_{\Om}(-\Psi_n^\prime(\varphi_t)+ \Psi_n^\prime(v)) (\varphi_t-v )~dx\\
	&  = \ld \varphi_t, \varphi_t-v \rd +  \int_{\Om}\Psi_n^\prime(\varphi_t)(\varphi_t-v )~dx - \ld v , \varphi_t-v \rd  -  \int_{\Om} \Psi_n^\prime(v) (\varphi_t-v )~dx\\
	&  \leq \ld \varphi_t, \varphi_t-v \rd +  \int_{\Om}\Psi_n^\prime(\varphi_t)(\varphi_t-v )~dx - ( F, \varphi_t -v )\\
	& =   \ld \varphi_t, \varphi_t-v -t\varphi \rd +  \int_{\Om}\Psi_n^\prime(\varphi_t)(\varphi_t-v -t\varphi)~dx - ( F, \varphi_t -v-t\varphi )\\
	& \hspace{4cm} + t\left(  \ld \varphi_t, \varphi \rd +  \int_{\Om}\Psi_n^\prime(\varphi_t)\varphi~dx - ( F, \varphi )\right)\\
	& \leq \ld w, \varphi_t-v -t\varphi \rd +  \int_{\Om}\Psi_n^\prime(w)(\varphi_t-v -t\varphi)~dx - ( F, \varphi_t -v-t\varphi )\\
	& \hspace{4cm}  + t\left(  \ld \varphi_t, \varphi \rd +  \int_{\Om}\Psi_n^\prime(\varphi_t)\varphi~dx - ( F, \varphi )\right)\\
	& \leq t\left(  \ld \varphi_t, \varphi \rd +  \int_{\Om}\Psi_n^\prime(\varphi_t)\varphi~dx - ( F, \varphi )\right). 
	\end{align*}
	Therefore, we obtain that 
	\begin{align*}
	\ld \varphi_t, \varphi \rd  +  \int_{\Om}\Psi_n^\prime(\varphi_t)\varphi~dx- ( F, \varphi ) &\geq \frac1t\left(\|\varphi_t-v\|^2 -   \int_{\Om}|\Psi_n^\prime(\varphi_t)- \Psi_n^\prime(v)| (\varphi_t-v )~dx \right)\\
	&\geq -\int_{\Om}|\Psi_n^\prime(\varphi_t)- \Psi_n^\prime(v)| \varphi~dx . 
	\end{align*}
		Now  using the weak convergence of $\varphi_t$ and 
		monotone convergence theorem, and dominated convergence theorem, we obtain 
		\begin{align}\label{si75}
		\ld v , \varphi \rd  \geq -  \int_{\Om}\Psi_n^\prime(v)\varphi~dx+ ( F, \varphi ) . 
		\end{align} 
		Using the density argument, one can easily show that \eqref{si75} is true for all $ \varphi \in X_0$ with $\varphi \geq 0$ a.e in $\Om$.  Note that  $v\geq 0$ implies  $\text{supp}(z-\epsilon-v)^+ \subset \text{supp}(z-\epsilon)^+)$ that is, $(z-v-\e)^+ \in X_0$.  So from \eqref{si75}, it implies  that 
	\begin{align}\label{si79}
	\ld v , (z-v-\e)^+ \rd  \geq -  \int_{\Om}\Psi_n^\prime(v)(z-v-\e)^+~dx+ ( F,(z-v-\e)^+ ) . 
	\end{align}
		Let $(z-v-\e)^+:= \mathfrak{g} \in X_0$ such that $0\leq \mathfrak{g}\leq z$ a.e in $\Om$. Let $\{\hat{\mathfrak{g}}_m \}$ be a monotonically increasing sequence in $C_c^\infty(\Om)$  such that $\{\hat{\mathfrak{g}}_m\}$  converging to $\mathfrak{g}$ in $X_0$ and set $\mathfrak{g}_m = \min\{ \hat{\mathfrak{g}}_m^+, \mathfrak{g} \}$.  Testing \eqref{si71} with $\mathfrak{g}_m$, we get 
		\begin{align}\label{si77}
		\ld z,\mathfrak{g}_m\rd \leq \int_{ \Om}z^{-q} \mathfrak{g}_m ~dx +( F,\mathfrak{g}_m)
		\end{align}
		Observe that if $\mathfrak{g} >0$ then $z >\e$. 
		Now consider 
		\begin{equation}\label{si93}
		\begin{aligned}
		\int_{\{\mathfrak{g} >0\}} & \int_{\{\mathfrak{g} >0\}} \frac{(z(x)-z(y))((\mathfrak{g}_m-\mathfrak{g})(x)-(\mathfrak{g}_m-\mathfrak{g})(y))}{|x-y|^{N+2s}}~dxdy\\
		& = \int_{\{\mathfrak{g} >0\}} \int_{\{\mathfrak{g} >0\}} \frac{((z-\e)^+(x)-(z-\e)^+(y))((\mathfrak{g}_m-\mathfrak{g})(x)-(\mathfrak{g}_m-\mathfrak{g})(y))}{|x-y|^{N+2s}}~dxdy\\
		& \leq \|(z-\e)^+\|~ \|(\mathfrak{g}_m-\mathfrak{g})\|\ra 0 \text{ as } m \ra \infty. 
		\end{aligned}
		\end{equation}
		\begin{equation}\label{si94}
		\begin{aligned}
		\int_{\R\setminus\{\mathfrak{g}>0\}}& \int_{\{\mathfrak{g}>0\}}\frac{(z(x)-z(y))(\mathfrak{g}_m(x)-\mathfrak{g}_m(y))}{|x-y|^{N+2s}}dxdy\\
		& =\int_{\R\setminus{\{\mathfrak{g}>0\}}}\int_{\{\mathfrak{g}>0\}}\frac{((z(x))(\mathfrak{g}_m(x))}{|x-y|^{N+2s}}dxdy -\int_{ \R\setminus\{\mathfrak{g}>0\}}\int_{\{\mathfrak{g}>0\}}\frac{(z(y))(\mathfrak{g}_m(x))}{|x-y|^{N+2s}}dxdy\\
		&  \geq\int_{ \R\setminus\{\mathfrak{g}>0\}}\int_{\{\mathfrak{g}>0\}}\frac{((z(x))(\mathfrak{g}_m(x))}{|x-y|^{N+2s}}dxdy -\int_{ \R\setminus\{\mathfrak{g}>0\}}\int_{\{\mathfrak{g}>0\}}\frac{(z(y))(\mathfrak{g}(x))}{|x-y|^{N+2s}}dxdy.
		\end{aligned}
		\end{equation}
		Taking into account the fact that $z^{-q} \mathfrak{g}_m \leq z^{-q} \mathfrak{g} $,   \eqref{si77}, \eqref{si93}, \eqref{si94} and monotone convergence theorem,    if $z^{-q} \mathfrak{g} \in L^1(\Om)$ or $z^{-q} \mathfrak{g}  \not \in L^1(\Om)$, we conclude that 
		\begin{align*}
		\ld z,\mathfrak{g}\rd \leq \int_{ \Om}z^{-q} \mathfrak{g} ~dx +( F,\mathfrak{g}).
		\end{align*}
	That is, 
	\begin{align}\label{si80}
	\ld z,(z-v-\e)^+\rd \leq \int_{ \Om}z^{-q} (z-v-\e)^+ ~dx +( F,(z-v-\e)^+)
	\end{align}
	Exploiting $n \geq \e^{-q}$, \eqref{si79},  \eqref{si80},   and the  fact that for any measurable  function $\mathfrak{h}$, $\ld 	\mathfrak{h}^+,\mathfrak{h}^+ \rd \leq \ld \mathfrak{h},\mathfrak{h}^+\rd$,  we obtain  that 
	\begin{align*}
	\ld (z-v-\e)^+,(z-v-\e)^+\rd &\leq \ld z-v,(z-v-\e)^+\rd\\
	& \leq  \int_{ \Om}(z^{-q}+ \Psi_n^\prime(v)) (z-v-\e)^+ ~dx \\
	& = \int_{ \Om}(-\Psi_n^\prime(z)+ \Psi_n^\prime(v)) (z-v-\e)^+ ~dx\leq 0. 
	\end{align*}
	Thus, $z\leq v+\e\leq w+\e$. Since $\e$ was arbitrary chosen, hence proof follows. \QED
\end{proof}
\section{Regularity and Proof of Theorem \ref{thmsi5}}
In this section, we start by extending some regularity results contained in \cite{newpaper} and conclude the proof of Theorem  \ref{thmsi5}. 
\begin{Lemma}\label{lemsi15}
	Any nonnegative solution to $(\widetilde{P_{\la}})$ belongs to $L^\infty(\Om)$. 
\end{Lemma}
\begin{proof}
		Let $u \in X_0$ be any  non negative weak solution to \eqref{si15}. Let  $u_\tau=   \min\{  u,\tau \}  $ for  $ \tau>0$.   Let   $ \phi = u(u_\tau)^{r-2}  \in X_0$ ($ r\geq 2$) be a test function to problem $(\widetilde{P_{\la}})$. Now from \cite[Lemma 3.5]{newpaper}, we have  the   following inequality
	\begin{align}\label{si85}
	\frac{4(r-1)}{r^2} \left(  a|a_k|^{\frac{r}{2}-1} -b|b_k|^{\frac{r}{2}-1} \right)^2 \leq 	(a-b)(a_k|a_k|^{r-2}-b_k|b_k|^{r-2}). 
	\end{align}
	where $a,b \in \mathbb{R}$ and $r \geq 2$.  Using \eqref{si85}, we deduce  that 
	\begin{equation}\label{si86}
	\begin{aligned}
	|u(u_\tau)^{\frac{r}{2}-1}|^2_{2^*_s}& \leq C \|u(u_\tau)^{\frac{r}{2}-1}\|^2  \leq \frac{Cr^2}{r-1} \int_{Q} \frac{(u(x)-u(y))(\phi(x)-\phi(y))}{|x-y|^{N+2s}}~dxdy\\
	&=  Cr \left(-g(x,u)u (u_\tau)^{r-2}~dx + \int_{ \Om}\int_{ \Om} \frac{(u+\overline{u})^{2^*_\mu}(u+\overline{u})^{2^*_\mu-1} u (u_\tau)^{r-2}}{|x-y|^{\mu}}~dxdy \right)\\
	& \leq Cr \int_{ \Om}\int_{ \Om} \frac{(u+\overline{u})^{2^*_\mu}(u+\overline{u})^{2^*_\mu-1} u (u_\tau)^{r-2}}{|x-y|^{\mu}}~dxdy \\
	&  \leq Cr  \left(\int_{ \Om}\int_{ \Om} \frac{u^{2^*_\mu}u^{2^*_\mu}  (u_\tau)^{r-2}}{|x-y|^{\mu}}~dxdy  + \int_{ \Om}\int_{ \Om} \frac{u^{2^*_\mu}\overline{u}^{2^*_\mu-1} u (u_\tau)^{r-2}}{|x-y|^{\mu}}~dxdy\right.\\
	& \left. \qquad +\int_{ \Om}\int_{ \Om} \frac{\overline{u}^{2^*_\mu}u^{2^*_\mu}  (u_\tau)^{r-2}}{|x-y|^{\mu}}~dxdy + \int_{ \Om}\int_{ \Om} \frac{\overline{u}^{2^*_\mu}\overline{u}^{2^*_\mu-1} u (u_\tau)^{r-2}}{|x-y|^{\mu}}~dxdy\right)\\
	& \leq Cr \left(|u|_{2^*_s}^{2^*_\mu} \left(\int_{ \Om} (u^{2^*_\mu} u_\tau^{r-2})^{\frac{2^*_s}{2^*_\mu}} \right)^{\frac{2^*_\mu}{2^*_s}}+ |u|_{2^*_s}^{2^*_\mu} |\overline{u}|_\infty^{2^*_\mu-1}\left(\int_{ \Om} (u u_\tau^{r-2})^{\frac{2^*_s}{2^*_\mu}} \right)^{\frac{2^*_\mu}{2^*_s}} \right. \\
	& \left.  \quad + |\overline{u}|_{2^*_s}^{2^*_\mu} \left(\int_{ \Om} (u^{2^*_\mu} u_\tau^{r-2})^{\frac{2^*_s}{2^*_\mu}} \right)^{\frac{2^*_\mu}{2^*_s}} + |\overline{u}|_{2^*_s}^{2^*_\mu} |\overline{u}|_\infty^{2^*_\mu-1} \left(\int_{ \Om} ( uu_\tau^{r-2})^{\frac{2^*_s}{2^*_\mu}} \right)^{\frac{2^*_\mu}{2^*_s}}	\right)
	\end{aligned}
	\end{equation} 
	\textbf{Claim:} Let $r_1= 2^*_s+1$. Then $ u \in L^{\frac{2^*_s r_1}{2}}(\Om)$.\\
	In view of H\"older's inequality, we have 
\begin{equation}\label{si87}
\begin{aligned}
\left(\int_{ \Om} (u^{2^*_\mu} u_\tau^{r-2})^{\frac{2^*_s}{2^*_\mu}} \right)^{\frac{2^*_\mu}{2^*_s}} &  = \left(\int_{ u\leq R} (u^{2^*_\mu} u_\tau^{r_1-2})^{\frac{2^*_s}{2^*_\mu}} \right)^{\frac{2^*_\mu}{2^*_s}}  +\left(\int_{ u>R} (|u|^{2^*_\mu} u_\tau^{r_1-2})^{\frac{2^*_s}{2^*_\mu}} \right)^{\frac{2^*_\mu}{2^*_s}} \\
	&  \leq R^{2^*_\mu} \left(\int_{ u\leq R} ( u_\tau^{r_1-2})^{\frac{2^*_s}{2^*_\mu}} \right)^{\frac{2^*_\mu}{2^*_s}}  +\left(\int_{ u>R} (u u_\tau^{r_1-2})^{\frac{2^*_s}{2}} \right)^{\frac{2}{2^*_s}} \\
	& \hspace{6cm}\left(\int_{ u>R} u^{2^*_s}\right)^{\frac{2^*_\mu-2}{2^*_s}}. 
	\end{aligned}
	\end{equation}	
		Choose $R>0$ large enough such that 
	\begin{align}\label{si88}
	\left( \int_{|u|>R} |u|^{2^*_s}  ~dx\right)^{\frac{2^*_s-2}{2^*_s}}  < \frac{1}{4Cr_1} \min\left\{ \frac{1}{|u|_{2^*_s}^{2^*_\mu}}, \frac{1}{|\overline{u}|_{2^*_s}^{2^*_\mu}}\right\}. 
	\end{align}
	Taking into account \eqref{si86}, \eqref{si87} jointly with \eqref{si88}, we obtain
	\begin{align*}
	|u(u_\tau)^{\frac{r_1}{2}-1}|^2_{2^*_s}& 
	\leq Cr \left(R^{2^*_\mu}|u|_{2^*_s}^{2^*_\mu} \left(\int_{ u\leq R} ( u^{2^*_s-1})^{\frac{2^*_s}{2^*_\mu}} \right)^{\frac{2^*_\mu}{2^*_s}}+ |u|_{2^*_s}^{2^*_\mu} |\overline{u}|_\infty^{2^*_\mu-1}\left(\int_{ \Om} u^{2^*_s} \right)^{\frac{2^*_\mu}{2^*_s}} \right. \\
	& \left.  \qquad + |\overline{u}|_{2^*_s}^{2^*_\mu}  R^{2^*_\mu} \left(\int_{ u\leq R} ( u^{2^*_s-1})^{\frac{2^*_s}{2^*_\mu}} \right)^{\frac{2^*_\mu}{2^*_s}} + |\overline{u}|_{2^*_s}^{2^*_\mu} |\overline{u}|_\infty^{2^*_\mu-1} \left(\int_{ \Om} u^{2^*_s}   \right)^{\frac{2^*_\mu}{2^*_s}}	\right).
	\end{align*}
	Appealing  Fatou's Lemma as $\tau \ra \infty$,  we obtain 
	\begin{align*}\label{ch41}
	||u|^{\frac{r_1}{2}}|^2_{2^*_s}& \leq Cr_1\left(|u|_{2^*_s}^{2^*_\mu} + |\overline{u}|_{2^*_s}^{2^*_\mu}\right) \left(R^{2^*_\mu} \left(\int_{ u\leq R} ( u^{2^*_s-1})^{\frac{2^*_s}{2^*_\mu}} \right)^{\frac{2^*_\mu}{2^*_s}}+  |\overline{u}|_\infty^{2^*_\mu-1}\left(\int_{ \Om} u^{2^*_s} \right)^{\frac{2^*_\mu}{2^*_s}} 	\right)	<\infty.
	\end{align*}
	This establishes the Claim. Now let $\tau \ra \infty$ in  \eqref{si86} and using the inequality $x^p<1+x$ for $p<1$ and $x \geq 0$  we obtain 
	\begin{equation*}
	\begin{aligned}
	||u|^{\frac{r}{2}}|^2_{2^*_s}& \leq Cr \left(|u|_{2^*_s}^{2^*_\mu} + |\overline{u}|_{2^*_s}^{2^*_\mu}\right) \left(\left(\int_{ \Om} (u^{2^*_\mu+r-2})^{\frac{2^*_s}{2^*_\mu}} \right)^{\frac{2^*_\mu}{2^*_s}}+  |\overline{u}|_\infty^{2^*_\mu-1}\left(\int_{ \Om} (u^{r-1})^{\frac{2^*_s}{2^*_\mu}} \right)^{\frac{2^*_\mu}{2^*_s}} \right)\\
	&	 \leq Cr \left(|u|_{2^*_s}^{2^*_\mu} + |\overline{u}|_{2^*_s}^{2^*_\mu}\right) \left(\left( 1+ \int_{ \Om} (u^{2^*_\mu+r-2})^{\frac{2^*_s}{2^*_\mu}} \right)+  |\overline{u}|_\infty^{2^*_\mu-1}\left(1+\int_{ \Om} (u^{r-1})^{\frac{2^*_s}{2^*_\mu}} \right) \right)\\
	& \leq 2Cr(1+|\overline{u}|_\infty^{2^*_\mu-1}+ |\Om|)\left(|u|_{2^*_s}^{2^*_\mu} + |\overline{u}|_{2^*_s}^{2^*_\mu}\right)  \left(1+ \int_{ \Om} (|u|^{r+2^*_\mu-2} )^{\frac{2^*_s}{2^*_\mu}}~dx \right).
	\end{aligned}
	\end{equation*}
	It implies 
	\begin{align}\label{si2}
	\left(1+  \int_{\Om } |u|^{ \frac{2^*_sr}{2} }~dx \right)^{\frac{2}{2^*(r-2)}}	   & 
	\leq   C_{r}^{\frac{1}{(r-2)}}  \left(1 +  \int_{ \Om}(u^{2^*_\mu-2+r } ) ^{\frac{2^*_s}{2^*_\mu}}~dx\right)^{\frac{1}{(r-2)}}
	\end{align}
	where $C_r = 4C r(1+|u|_{2^*_s}^{2^*_\mu}+|\Om|)\left(|u|_{2^*_s}^{2^*_\mu} + |\overline{u}|_{2^*_s}^{2^*_\mu}\right)$. For $j\geq 1$ we define $r_{j+1}$ inductively as 
	\begin{align*}
	(r_{j+1} + 2^*_\mu -2) \frac{2^*_s}{2^*_\mu}= \frac{2^*_sr_j}{2}. 
	\end{align*}
	That is, $(r_{j+1}-2)= \left(\frac{2^*_\mu}{2}\right)^j (r_1-2)$. From \eqref{si2} with $   C_{r_{j+1}} =  4C r_{j+1}(1+|u|_{2^*_s}^{2^*_\mu}+|\Om|)$,  it follows that 
	\begin{align*}
	\left(1+  \int_{\Om }|u|^{ \frac{2^*_sr_{j+1}}{2} }~dx \right)^{\frac{2}{2^*_s(r_{j+1}-1)}}	   & 
	\leq   C_{r_{j+1}}^{\frac{1}{(r_{j+1}-2)}}   \left(1 +  \int_{ \Om}(u^{2^*_\mu-2+r_{j} } ) ^{\frac{2^*_s}{2^*_\mu}}~dx\right)^{\frac{2}{2^*_s(r_{j}-2)}}. 
	\end{align*}
	Defining $A_j:= \left(1 +  \int_{ \Om}(u^{2^*_\mu-2+r_{j} } ) ^{\frac{2^*_s}{2^*_\mu}}~dx\right)^{\frac{2}{2^*_s(r_{j}-2)}}$. Then by Claim  and limiting argument, there exists $C_0>0$ such that 
	\begin{align*}
	A_{j+1}\leq \prod_{k=2}^{j+1} C_k^{(1/2(r_k-1))} A_1\leq C_0A_1. 
	\end{align*}
	Hence $|u|_{\infty}\leq C_0A_1$. That is $ u \in L^\infty(\Om)$. 	\QED
\end{proof}

\begin{Remark}
	We remark that if  $u \in X_0$ be any  weak solution of the following problem
		\begin{equation}\label{si15}
		(-\De)^s u
		=f(x,u)+  \left(\ds \int_{\Om}\frac{|u|^{2^*_\mu}(y)}{|x-y|^{\mu}}dy\right) |u|^{2^*_\mu-1}  \; \text{in}\;
		\Om,
		u=0 \; \text{ in } \R \setminus \Om,
		\end{equation}
		where $|f(x,u))|\leq C(1+ |u|^{2^*-1})$ and $\mu \leq \min\{ 4s,N\}$.  Then by using the same assertions as in Lemma \ref{lemsi15}, we obtain that   $ u \in L^\infty(\Om)$. This complements in the singular case previous results proved in \cite{newpaper}.
\end{Remark}

\begin{Lemma}\label{lemsi19}
	Let  $ z \in L^{2^*_s}(\Om)$ be a positive function,  let $ h(x,z) = \left(\ds \int_{\Om}\frac{z^{2^*_\mu}(y)}{|x-y|^{\mu}}dy\right) z^{2^*_\mu-1}$. Assume  $u \in X_0$ be a positive weak solution to 
	\begin{align}\label{si82}
	(-\De)^s u +g(x,u)= h(x,z) \text{ in } \Om, \qquad u=0 \text{ in } \R\setminus \Om. 
	\end{align}
	Then $(u+\overline{u} -\e)^+ \in X_0$ for every $\e>0$. 	
\end{Lemma}
\begin{proof}
	Using the assertions and arguments used in \cite[Lemma 3.4]{GMS}, one can easily proof the result, we leave it for the readers. \QED
\end{proof}
\begin{Lemma}\label{lemsi21}
	Let $\la>0$ and let $z  \in H^s_{\text{loc}}(\Om)\cap L^{2^*_s}(\Om)$ be a weak solution to $(P_\la)$ as it is defined in   definition \ref{defisi1}. Then $z- \overline{u}$ is a positive weak solution to $(\widetilde{P_{\la}})$ belonging to $L^\infty(\Om)$. 
\end{Lemma}
\begin{proof}
	Consider problem  \eqref{si82} with $z$ given. Then $0$ is a strict subsolution to \eqref{si82}.  Define the functional $I : X_0 \ra (-\infty, \infty]$ by 
	\begin{align*}
	I(u)= 
	\left\{
	\begin{array}{ll}
	\frac12 \|u\|^2 + \int_{\Om}G(x,u)~dx - \frac{\la}{22^*_\mu} \iint_{\Om\times \Om}\frac{z^{2^*_\mu}z^{2^*_\mu-1}u }{|x-y|^{\mu}}~dxdy &  \text{ if }  G(\cdot, u ) \in L^1(\Om) , \\
	\infty  & \text{ otherwise. }  
	\end{array} 
	\right.
	\end{align*}
	Moreover for the closed convex set $K_0 = \{u \in X_0~:~ u \geq 0 \}$ we define 	$I_{K_0} : X_0 \ra  (-\infty, \infty]$ by 
	\begin{align*}
	I_{K_0}(u)= 
	\left\{
	\begin{array}{ll}
	I(u)  &  \text{ if }   u \in K_0 \text{ and }G(\cdot, u ) \in L^1(\Om) , \\
	\infty  & \text{ otherwise. }  
	\end{array} 
	\right.
	\end{align*}
	we  can easily prove that there exists $u \in K_0$ such that $I_{K_0}(u)=  \inf I_{K_0}(K_0)$.  It implies that $0 \in \pa^- I_{K_0} (u)$. Now from Proposition \ref{propsi2}, we obtain that $u$ is a non negative solution to \eqref{si82}.  Using the Lemma \ref{lemsi19}, Lemma \ref{lemsi20} and  assertions as in Lemma \ref{lemsi18}, we obtain that  $(u+\overline{u}-\e)^+ \in X_0$ for every $\e>0$ and
	\begin{align*}
	& \ld u+\overline{u}, v \rd  - \int_{ \Om} (u+\overline{u})^{-q}v~dx - \iint_{\Om\times \Om}\frac{z^{2^*_\mu}z^{2^*_\mu-1}v}{|x-y|^{\mu}}~dxdy=0\\
	& 	\ld z, v \rd  - \int_{ \Om} z^{-q}v~dx - \iint_{\Om\times \Om}\frac{z^{2^*_\mu}z^{2^*_\mu-1}v}{|x-y|^{\mu}}~dxdy=0
	\end{align*}
	for all  compactly supported $0\leq v \in X_0 \cap L^\infty(\Om)$. 
	To prove the above equations for all  compactly supported $0\leq v \in X_0 \cap L^\infty(\Om)$ one can use the fact that $u \in X_0,~ \overline{u} \in H^s_{\text{loc}}(\Om)$ (See Remark \ref{remsi1})  and the assertions as in Lemma \ref{lemsi20} and Lemma \ref{lemsi18}.	Now using the Lemma \ref{lemsi17},  we get $z= u+\overline{u}$. That $u = z-\overline{u}$ is a solution to $(\widetilde{P_{\la}})$. And from Lemma \ref{lemsi15}, we have $ u \in L^\infty(\Om)$. \QED
\end{proof}
\begin{Lemma}\label{lemsi2}
	Let $\mu \leq \min\{ 4s,N\}$. 	Let $u $ be any  weak solution of problem $(P_\la)$. Then $ u \in L^\infty(\Om)\cap C^+_{\phi_q}(\Om)\cap C^\ga(\R) $ where $\ga$ is defined \eqref{si4}.
\end{Lemma}
\begin{proof}
	Let $u$ be any  weak solution of problem $(P_\la)$. Employing  Lemma \ref{lemsi21},  $u-\overline{u} \in X_0$ is the solution to $(\widetilde{P_{\la}})$ and which on taking account Lemma \ref{lemsi5}, we have  $u-\overline{u} \in L^\infty(\Om)$. Therefore, $u = (u-\overline{u})+ \overline{u} \in L^\infty(\Om)$. 
	Let  $\hat{u}$ be a  unique solution (See \cite[Theorem 1.2, Remark 1.5]{adi}) to the following problem
	\begin{align*}
	(-\De)^s \hat{u} = \hat{u}^{-q}+\la c, u>0 \text{ in } \Om, \hat{u} =0 \text{ in } \R\setminus \Om
	\end{align*}
	where $c= C^*|u|_\infty^{22^*_\mu-1} $ with $C^*= \bigg|\ds \int_{\Om}\frac{dy}{|x-y|^{\mu}}\bigg|_\infty$.   Practising Lemma \ref{lemsi17}, one  can easily show that $\overline{u}\leq u \leq \hat{u}$ a.e in $\Om$.  
	Now using the fact that  $\overline{u}\leq u \leq \hat{u}$ a.e in $\Om$ and regularity of $\overline{u}$ and $\hat{u}$ we obtain $ u \in C^+_{\phi_q}(\Om)$.  Observe that $u$ is a classical solution in sense of \cite[Definition 1]{adi} so by  \cite[Theorem 1.2]{adi},  H\"older's regularity follows. 
	\QED
\end{proof}
{\bf Proof of Theorem \ref{thmsi5}}: It follows from the proof of Lemma \ref{lemsi2} and of Lemma \ref{lemsi16}.\QED

\section{Existence of first solution}
 In this section, we have prove the existence of first solution and further establish that the first solution is actually a local minimizer of an appropriate functional.  We start the section by defining the functional associated with $(\widetilde{P_{\la}})$. 
 Consider  the functional $\mc J : X_0 \ra (-\infty, \infty]$  associated with 
 \begin{align*}
 \mc J(u)= 
 \left\{
 \begin{array}{ll}
 \frac12 \|u\|^2 + \int_{\Om}G(x,u)~dx - \frac{\la}{22^*_\mu} \iint_{\Om\times \Om}\frac{|u|^{2^*_\mu}|u|^{2^*_\mu}}{|x-y|^{\mu}}~dxdy &  \text{ if }  G(\cdot, u ) \in L^1(\Om) , \\
 \infty  & \text{ otherwise. }  
 \end{array} 
 \right.
 \end{align*}
For any convex subset $K \subset X_0$ we define the functional  $\mc J_K : X_0 \ra  (-\infty, \infty]$ by 
\begin{align*}
\mc J_K(u)= 
\left\{
\begin{array}{ll}
\mc J(u)  &  \text{ if }   u \in K \text{ and }G(\cdot, u ) \in L^1(\Om) , \\
\infty  & \text{ otherwise. }  
\end{array} 
\right.
\end{align*}
Define $\La:= \sup\{ \la>0\; : \; (P_\la) \text{ has a weak solution} \} $. 
\begin{Lemma}\label{lemsi14}
	Let $K$ be a convex subset of $X_0$ and let $w \in X_0$. Let $ u \in K $ with $G(\cdot, u) \in L^1(\Om)$. Then the following assertions are equivalent:
	\begin{enumerate}
		\item [(i)]	$ \al \in \pa ^- \mc J_K (u)$.
		\item [(ii)] For every $ w  \in K$ with $G(\cdot , w) \in L^1(\Om)$, we have $g(\cdot,u )(w-u) \in L^1(\Om)$  and 
		\begin{align*}
	\ld \al , w- u \rd & \leq 	\ld u, (w-u)\rd + \int_{\Om} g(x,u)(w-u)~dx\\
	&  -\la \iint_{\Om\times \Om}\frac{(u+\overline{u})^{2^*_\mu}(u+\overline{u})^{2^*_\mu-1}(w-u)}{|x-y|^{\mu}}~dxdy. 
		\end{align*}
	\end{enumerate}	
\end{Lemma}

\begin{proof}
	(i) implies (ii).  Let $ w \in K $ and $G(\cdot, w ) \in L^1(\Om) $. Define  $z= w-u$. Then clearly since $g(x,u)$ is increasing in $u$, we have $g(x,u)z \leq G(x,w)- G(x,u)$. Moreover, $(g(\cdot,u)z)\vee 0 \in L^1(\Om)$ and $t \mapsto (G(x,u+tz)- G(x,u))/t, (0,1] \ra \mathbb{R}$, is increasing  and 
	\begin{align*}
	\frac{\mc J_K(u+tz)- \mc J_K(u)}{t} & = \ld u,w\rd +\frac{t\|z\|^2}{2} + \int_{ \Om}\frac{(G(x,u+tz)- G(x,u))}{t}\\
	& \quad - \frac{1}{22^*_\mu t }  \iint_{\Om\times \Om}\frac{(u+\overline{u}+tz)^{2^*_\mu}(u+\overline{u}+tz)^{2^*_\mu}}{|x-y|^{\mu}}~dxdy \\
	& \quad  + \frac{1}{22^*_\mu t } \iint_{\Om\times \Om}\frac{(u+\overline{u})^{2^*_\mu}(u+\overline{u})^{2^*_\mu}}{|x-y|^{\mu}}~dxdy. 
	\end{align*}
	Passing to the limit as $t \ra 0$ and using the fact that $\al \in \pa^-\mc J_K(u)$, we deduce the required result. 
	(ii) implies (i). Let $z \in K$ and $G(\cdot, w ) \in L^1(\Om)$. Employing the fact that $G(x,s)$ is convex is $s$ and using  (ii) we have that 
	\begin{align*}
	\mc J_K(w)- \mc J_K(u)&  = \frac{1}{2} \|z\|^2  +\int_{ \Om}(G(x,w)-G(x,u)- g(x,u)z)~dx + \ld \al, z\rd \\
	& \quad  - \frac{\la}{22^*_\mu} \iint_{\Om\times \Om}\frac{\left( (w+\overline{u})^{2^*_\mu}(w+\overline{u})^{2^*_\mu} - (u+\overline{u})^{2^*_\mu}(u+\overline{u})^{2^*_\mu}\right) }{|x-y|^{\mu}}~dxdy\\
	& \quad + \la \iint_{\Om\times \Om}\frac{(u+\overline{u})^{2^*_\mu}(u+\overline{u})^{2^*_\mu-1}z}{|x-y|^{\mu}}~dxdy.  
	\end{align*} 
	It implies that $ \al \in \pa^-\mc J_K (u)$. \QED
\end{proof}
	For any functions $\varphi, \psi : \Om \ra [-\infty,+\infty]$, we define the following subspaces 
	\begin{align*}
	K_\varphi= \{ u \in X_0\; : \varphi \leq u \text{ a.e} \}, K^\psi= \{ u \in X_0\; : u \leq \psi \text{ a.e} \}, K_\varphi^\psi= \{ u \in X_0\; : \varphi \leq u \leq \psi\text{ a.e} \}. 
	\end{align*}
	
	\begin{Proposition}\label{propsi2}
		Assume one the following condition holds:
		\begin{enumerate}
			\item [(i)] $\phi_1$ is  a subsolution to $(\widetilde{P_\la}), \; G(x,w(x)) \in L^1_{\text{loc}}(\Om)$ for all $ w \in K_{\phi_1}, u \in D(\mc J_{K_{\phi_1}})$ and $0 \in \pa^-\mc J_{K_{\phi_1}}(u)$.
				\item [(ii)] $\phi_2$ is  a supersolution to $(\widetilde{P_\la}), \; G(x,w(x)) \in L^1_{\text{loc}}(\Om)$ for all $ w \in K^{\phi_2}, u \in D(\mc J_{K^{\phi_2}})$ and $0 \in \pa^-\mc J_{K^{\phi_1}}(u)$.
				\item [(iii)] 	 $\phi_1, \phi_2$ are  subsolution  and supersolution to $(\widetilde{P_\la}),\; \phi_1\leq \phi_2,  \; G(x, \phi_1), G(x, \phi_2) \in L^1_{\text{loc}}(\Om)$, $ u \in D(\mc J_{K_{\phi_1}^{\phi_2}})$ and $0 \in \pa^-\mc J_{K_{\phi_1}^{\phi_2}}(u)$.
		\end{enumerate}
	Then $u$ is weak solution to $(\widetilde{P_\la})$. 
	\end{Proposition}
	\begin{proof}
		Follow the \cite[Proposition 4.2]{GMS}, we have the required result. \QED
	\end{proof}
Let $\vartheta \in C^s(\R)\cap  X_0$ be the  unique solution which satisfies $(-\De)^s\vartheta = 1/2$ in $\Om$ in the sense of distributions. By the definition of $g$ and $G$, we obtain the following properties
\begin{Lemma}\label{lemsi10}
	\begin{enumerate}
		\item[(i)]Let $ u \in L^1_{\text{loc}}(\Om)$ such that $ \ds \text{ess inf}_K
		u >0 $ for any compact set $K \subset \Om$. Then $g(x,u(x)), G(x,u(x)) \in L^1_{\text{loc}}(\Om)$. 
		\item [(ii)] For all $x \in \Om$, the following holds
		\begin{itemize}
			\item [(a)] $G(x,st) \leq s^2G(x,t)$ for each $s \geq 1$ and $t\geq 0$. 
			\item [(b)] $G(x,s)-G(x,t) -(g(x,s)+g(x,t))(s-t)/2\geq 0 $ for each $s,t $ with $s\geq t > -\vartheta(x)$. 
			\item [(c)] $G(x,s)- g(x,s)s/2\geq 0 $ for each $s \geq 0 $. 
		\end{itemize}
	\end{enumerate}
	
\end{Lemma}
\begin{Lemma}\label{lemsi12}
	The following hold:
	\begin{enumerate}
		\item [(i)] $0$ is the strict subsolution to $(\widetilde{P_\la})$ for all $\la>0$. 
		\item [(ii)] $\vartheta$ is a strict supersolution to $(\widetilde{P_\la})$ for all sufficiently small $\la>0$. 
		\item [(iii)]  Any positive  weak solution $w$ to $(\widetilde{P_{\la_2}})$ is a strict supersolution to $(\widetilde{P_{\la_1}})$ for $0<\la_1<\la_2$. 
	\end{enumerate} 
\end{Lemma}
	\begin{proof}
		(i) Trivial.\\
		(ii) Choose $\la$ small enough such that $\la\left(\ds\int_{\Om}\frac{(\vartheta+\overline{u})^{2^*_\mu}}{|x-y|^{\mu}}~dy\right)(\vartheta+\overline{u})^{2^*_\mu-1}<1$ in $\Om$. From Lemma  \ref{lemsi10}, $g(x,\vartheta), \; G(x,\vartheta)  \in  L^1_{\text{loc}}(\Om)$,  for all $\psi \in X_0 \setminus \{0\}$,  we deduce that 
		\begin{align*}
		& \ld \vartheta, \psi \rd + \int_{ \Om} g(x,\vartheta)\psi ~dx -\la \int_{\Om}\int_{\Om}\frac{(\vartheta+\overline{u})^{2^*_\mu}(\vartheta+\overline{u})^{2^*_\mu-1}\psi}{|x-y|^{\mu}}~dxdy\\
		& \geq \int\left(1-    \la \left(\ds\int_{\Om}\frac{(\vartheta+\overline{u})^{2^*_\mu}}{|x-y|^{\mu}}~dy\right)(\vartheta+\overline{u})^{2^*_\mu-1}\right)  \psi~dx >0. 
		\end{align*}
		(iii) Let  $0<\la_1<\la_2$ and $w$ be a positive weak solution to $(\widetilde{P_{\la_2}})$.  Then for all $\psi \in X_0 \setminus \{0\}$, we have 
		\begin{align*}
	& 	\ld w, \psi \rd +   \int_{ \Om} g(x,w)\psi ~dx -\la_1 \int_{\Om}\int_{\Om}\frac{(w+\overline{u})^{2^*_\mu}(w+\overline{u})^{2^*_\mu-1}\psi}{|x-y|^{\mu}}~dxdy\\
		& = (\la_2-\la_1) \int_{\Om}\int_{\Om}\frac{(w+\overline{u})^{2^*_\mu}(w+\overline{u})^{2^*_\mu-1}\psi}{|x-y|^{\mu}}~dxdy>0. 
		\end{align*}
		The proof is now complete. \QED
	\end{proof}
\begin{Theorem}\label{thmsi2}
	Let $ w_1, w_2 :\Om \ra [-\infty, + \infty]$ with $w_1 \leq w_2$ such that $w_1 $ is a strict subsolution to $(\widetilde{P_\la})$ and $u \in D(\mc J_{K_{w_1}^{w_2}})$ be a minimizer for $\mc J_{K_{w_1}^{w_2}}$. Then $u$ is a local minimizer for $\mc J_{K_{w_1}}$. 
\end{Theorem}
\begin{proof}
	For each $v \in K_{w_1}$ and $0\leq \phi \in X_0$, we define $\sigma(v)= \min \{ v,w_2\}= v- (v-w_2)^+$ and 
\begin{align*}
\varXi(\phi)= \ld w_2, \phi\rd  +   \int_{ \Om} g(x,w_2)\phi ~dx -\la \int_{\Om}\int_{\Om}\frac{(w_2+\overline{u})^{2^*_\mu}(w_2+\overline{u})^{2^*_\mu-1}\phi}{|x-y|^{\mu}}~dxdy. 
\end{align*}  
\textbf{Claim}:   $\ld \sigma(v),  v-\sigma(v)\rd\geq  \ld w_2,  v-\sigma(v)\rd$ and
\begin{align*}
 \int_{\Om}\int_{\Om}\frac{\left( (\sigma(v)+\overline{u})^{2^*_\mu}(\sigma(v)+\overline{u})^{2^*_\mu-1} - (w_2+\overline{u})^{2^*_\mu}(w_2+\overline{u})^{2^*_\mu-1}\right)   (v-\sigma(v))}{|x-y|^{\mu}}~dxdy \leq 0 . 
\end{align*}
Notice that $v-\sigma(v)= (v-w_2)^+$. Let $\Om_1= \text{supp}((v-w_2)^+)$. Then on $\Om_1, \sigma(v)= w_2$ and using the fact that $\sigma(v) \leq w_2$ on $\Om$, we have 
\begin{align*}
\ld \sigma(v),  v-\sigma(v)\rd & =  \left( 
\int_{\Om_1}\int_{\Om_1}  + 2\int_{\R\setminus \Om_1}\int_{\Om_1}  + \int_{\Om\setminus \Om_1}\int_{\Om \setminus \Om_1} + 2\int_{\R\setminus \Om}\int_{\Om\setminus \Om_1} \right.\\
& \left.    \quad \frac{(\sigma(v)(x)-\sigma(v)(y))((v-\sigma(v))(x)-(v-\sigma(v))(y))}{|x-y|^{N+2s}}~dxdy\right)
\\& \geq  \ld w_2,  v-\sigma(v)\rd . 
\end{align*}
Second holds by using the  fact that $\sigma(v) \leq w_2$ on $\Om$.  It implies that the Claim holds. 
Taking into account the fact that $u$ is a minimizer of  for $\mc J_{K_{w_1}^{w_2}},\; \sigma(v) \in D(\mc J_{K_{w_1}^{w_2}})$, Lemma 2 of \cite{hirano} and the fact that $G(x,\cdot )$ is convex, we infer that 
\begin{equation}\label{si41}
\begin{aligned}
& \mc J_{K_{w_1}}(v)- \mc J_{K_{w_1}}(u)  \geq  \mc J_{K_{w_1}}(v)- \mc J_{K_{w_1}}(\sigma(v))\\
&  = \frac{\|v-\sigma(v)\|^2}{2} +\ld \sigma(v), v-\sigma(v)\rd + \int_{ \Om}(G(x,v)-G(x,\sigma(v)))~dx \\
& \qquad -\frac{\la}{22^*_\mu} \int_{\Om}\int_{\Om}\frac{ \left( (v+\overline{u})^{2^*_\mu}(v+\overline{u})^{2^*_\mu} - (\sigma(v)+\overline{u})^{2^*_\mu}(\sigma(v)+\overline{u})^{2^*_\mu}\right)}{|x-y|^{\mu}}~dxdy\\
& \geq  \frac{\|v-\sigma(v)\|^2}{2} +\ld \sigma(v),  v-\sigma(v)\rd + \int_{ \Om}g(x,\sigma(v))(v-\sigma(v))~dx \\
& \qquad -\frac{\la}{22^*_\mu} \int_{\Om}\int_{\Om}\frac{ \left( (v+\overline{u})^{2^*_\mu}(v+\overline{u})^{2^*_\mu} - (\sigma(v)+\overline{u})^{2^*_\mu}(\sigma(v)+\overline{u})^{2^*_\mu}\right)}{|x-y|^{\mu}}~dxdy\\
& \geq  \frac{\|v-\sigma(v)\|^2}{2} +\ld w_2,  v-\sigma(v)\rd + \int_{ \Om}g(x,w_2)(v-\sigma(v))~dx \\
& \qquad -\frac{\la}{22^*_\mu} \int_{\Om}\int_{\Om}\frac{ \left( (v+\overline{u})^{2^*_\mu}(v+\overline{u})^{2^*_\mu} - (\sigma(v)+\overline{u})^{2^*_\mu}(\sigma(v)+\overline{u})^{2^*_\mu}\right)}{|x-y|^{\mu}}~dxdy\\
&
\geq  \frac{\|v-\sigma(v)\|^2}{2} + \varXi(v-\sigma(v)) -\frac{\la}{22^*_\mu}  I  
\end{aligned}
\end{equation}
where 
\begin{align*}
I=&  \int_{\Om}\int_{\Om}\frac{(v+\overline{u})^{2^*_\mu}(v+\overline{u})^{2^*_\mu}}{|x-y|^{\mu}}~dxdy    - \int_{\Om}\int_{\Om}\frac{(\sigma(v)+\overline{u})^{2^*_\mu}(\sigma(v)+\overline{u})^{2^*_\mu}}{|x-y|^{\mu}}~dxdy \\ &   \quad - 22^*_\mu \int_{\Om}\int_{\Om}\frac{(\sigma(v)+\overline{u})^{2^*_\mu}(\sigma(v)+\overline{u})^{2^*_\mu-1}(v-\sigma(v))}{|x-y|^{\mu}}~dxdy.
\end{align*}
Now we estimate $I$ from above. First observe that 
\begin{equation}
\begin{aligned}\label{si42}
I& = 2^*_\mu \int_{\Om}  \int_{\sigma(v)}^{v} \left( \int_{\Om}\frac{(v+\overline{u})^{2^*_\mu} + (\sigma(v)+\overline{u})^{2^*_\mu}  }{|x-y|^{\mu}}~dy\right) \left( (t+\overline{u})^{2^*_\mu-1}- (\sigma(v)+\overline{u})^{2^*_\mu-1}\right)~dtdx\\
& \quad+  2^*_\mu \int_{\Om}  \int_{\sigma(v)}^{v} \left( \int_{\Om}\frac{(v+\overline{u})^{2^*_\mu} - (\sigma(v)+\overline{u})^{2^*_\mu} }{|x-y|^{\mu}}~dy\right)  (\sigma(v)+\overline{u})^{2^*_\mu-1} ~dtdx. 
\end{aligned}
\end{equation}
Using the mean value theorem, there exists $\theta\in [0,1]$ such that 
\begin{align*}
\frac{(u+\overline{u})^{2^*_\mu-1}- (v+\overline{u})^{2^*_\mu-1}}{(u-v)}& = (2^*_\mu-1) (u+\overline{u} + \theta(v-u))^{2^*_\mu-2}(u-v) \\
&=  (2^*_\mu-1) (\overline{u} + (1-\theta)u + \theta v)^{2^*_\mu-2}(u-v)\\
& \leq (2^*_\mu-1)2^{2^*_\mu-3} (\overline{u}^{2^*_\mu-2} +( (1-\theta)u + \theta v)^{^{2^*_\mu-2}})(u-v)\\
& \leq (2^*_\mu-1)2^{2^*_\mu-3} (\overline{u}^{2^*_\mu-2} +\max\{u,v  \}^{^{2^*_\mu-2}})(u-v). 
\end{align*}
For each $x \in \Om$ and $v \in D(\mc J_{K_{w_2}})$ define the functions
\begin{align*}
& k_v^1(x) = (2^*_\mu-1)2^{2^*_\mu-3} (\overline{u}^{2^*_\mu-2} +\max\{|w_2|,|v|  \}^{^{2^*_\mu-2}})\chi_{\{v>w_2\}}, \\
& k_v^2(x) = 2^*_\mu2^{2^*_\mu-2} (\overline{u}^{2^*_\mu-1} +\max\{|w_2|,|v|  \}^{^{2^*_\mu-1}}) \chi_{\{v>w_2\}}.
\end{align*}
Using the Hardy-Littlewood-Sobolev inequality, we have 
\begin{equation}\label{si43}
\begin{aligned}
&\int_{\Om}  \int_{\sigma(v)}^{v} \left( \int_{\Om}\frac{(v+\overline{u})^{2^*_\mu} + (\sigma(v)+\overline{u})^{2^*_\mu}  }{|x-y|^{\mu}}~dy\right) \left( (t+\overline{u})^{2^*_\mu-1}- (\sigma(v)+\overline{u})^{2^*_\mu-1}\right)~dtdx\\
& \leq \frac12 \int_{\Om}    \int_{\Om}\frac{((v+\overline{u})^{2^*_\mu} + (\sigma(v)+\overline{u})^{2^*_\mu})   k_v^1(x) (v-\sigma(v))^2 }{|x-y|^{\mu}}~dy dx\\
& \leq c_1 \left( | v+\overline{u}|_{2^*_s}^{2^*_\mu}+ | \sigma(v)+\overline{u}|_{2^*_s}^{2^*_\mu}\right)| k_v^1(x) (v-\sigma(v))^2|_{\frac{2^*_s}{2^*_\mu}}
\end{aligned}
\end{equation}
 for some appropriate  positive constant $c_1$. Similarly with the help of Hardy-Littlewood-Sobolev inequality, H\"older's inequality and the definition of $S$ we have  
\begin{equation}\label{si44}
\begin{aligned}
& \int_{\Om}  \int_{\sigma(v)}^{v} \left( \int_{\Om}\frac{(v+\overline{u})^{2^*_\mu} - (\sigma(v)+\overline{u})^{2^*_\mu} }{|x-y|^{\mu}}~dy\right)  (\sigma(v)+\overline{u})^{2^*_\mu-1} ~dtdx  \\
& \leq c_2 S^{-1/2} | k_v^2(x) (v-\sigma(v))|_{\frac{2^*_s}{2^*_\mu}} | \sigma(v)+\overline{u}|_{2^*_s}^{2^*_\mu-1} \|v-\sigma(v)\|
\end{aligned}
\end{equation}
for some appropriate  positive constant $c_1$. Using \eqref{si42} jointly with \eqref{si43} and \eqref{si44}, we have 
\begin{equation}\label{si45}
\begin{aligned}
I \leq & c_1 \left( | v+\overline{u}|_{2^*_s}^{2^*_\mu}+ | \sigma(v)+\overline{u}|_{2^*_s}^{2^*_\mu}\right)| k_v^1(x) (v-\sigma(v))^2|_{\frac{2^*_s}{2^*_\mu}}\\
& \quad +c_2 S^{-1/2} | k_v^2(x) (v-\sigma(v))|_{\frac{2^*_s}{2^*_\mu}} | \sigma(v)+\overline{u}|_{2^*_s}^{2^*_\mu-1} \|v-\sigma(v)\|. 
\end{aligned}
\end{equation}
Let us suppose that the result is not true. This means that there exists a sequence $\{ v_n \} \subset X_0$ such that for any $v_n \in K_{w_1}$ and 
\begin{align*}
\|v_n- u\| < \frac{1}{2^n},\; \mc J_{K_{w_1}} (v_n) < \mc J_{K_{w_1}} (u) \text{ for all } n. 
\end{align*} 
Define $ l:= u + \sum_{n=1}^{\infty} |v_n - u|$. By definition, $|v_n|\leq l $ a.e for all $n$.  Now  for each $v \in D(\mc J_{K_{w_1}})$, set
\begin{align*}
& \underline{k_v^1}(x) = (2^*_\mu-1)2^{2^*_\mu-3} (\overline{u}^{2^*_\mu-2} +\max\{|w_2|,|l|  \}^{^{2^*_\mu-2}}) \chi_{\{v>w_2\}}\\
&\underline{k_v^2}(x) = 2^*_\mu2^{2^*_\mu-2} (\overline{u}^{2^*_\mu-1} +\max\{|w_2|,|l|  \}^{^{2^*_\mu-1}}) \chi_{\{v>w_2\}}. 
\end{align*}
Employing \eqref{si41} and \eqref{si45}, we deduce that 
\begin{equation}\label{si46}
\begin{aligned}
0& > \mc J_{K_{w_1}}(v_n)- \mc J_{K_{w_1}}(u)  \\
& \geq  \mc J_{K_{w_1}}(v_n)- \mc J_{K_{w_1}}(\sigma(v_n))\\
& \geq  \frac{\|v_n-\sigma(v_n)\|^2}{2}  -\la\left(  c_1 \left( | v_n+\overline{u}|_{2^*_s}^{2^*_\mu}+ | \sigma(v_n)+\overline{u}|_{2^*_s}^{2^*_\mu}\right)| \underline{k_{v_n}^1}(x) (v_n-\sigma(v_n))^2|_{\frac{2^*_s}{2^*_\mu}}\right. \\
& \left.\hspace{1cm}+c_2 S^{-1/2} | \underline{k_{v_n}^2} (x) (v_n-\sigma(v_n))|_{\frac{2^*_s}{2^*_\mu}} | \sigma(v_n)+\overline{u}|_{2^*_s}^{2^*_\mu-1} \|v_n-\sigma(v_n)\|\right) + \varXi(v_n-\sigma(v_n))\\
& \geq  \frac{\|v_n-\sigma(v_n)\|^2}{2} + \varXi(v_n-\sigma(v_n)) -\left(  \frac{C_1}{4} | \underline{k_{v_n}^1}(x) (v_n-\sigma(v_n))^2|_{\frac{2^*_s}{2^*_\mu}}\right. \\
& \left.\hspace{5cm}+\frac{C_2}{4} | \underline{k_{v_n}^2} (x) (v_n-\sigma(v_n))|_{\frac{2^*_s}{2^*_\mu}}  \|v_n-\sigma(v_n)\|\right)
\end{aligned}
\end{equation}
where $C_1 = \sup_{n} 4\la c_1 \left( | v_n+\overline{u}|_{2^*_s}^{2^*_\mu}+ | \sigma(v_n)+\overline{u}|_{2^*_s}^{2^*_\mu}\right)$ and $C_2 = \sup_{n} 4\la c_2 S^{-1/2}| \sigma(v_n)+\overline{u}|_{2^*_s}^{2^*_\mu-1}$. 
Consider 
\begin{align*}
| \underline{k_{v_n}^1}(x)& (v_n-\sigma(v_n))^2|_{\frac{2^*_s}{2^*_\mu}}\leq  | \underline{k_{v_n}^1}(x) |_{\frac{2^*_s}{2^*_\mu-2}} |  (v_n-\sigma(v_n))^2|_{\frac{22^*_s}{2^*_\mu}}^2\\
& = \left( \left(\int_{ \{ \underline{k_{v_n}^1} \leq R_1\}} | \underline{k_{v_n}^1}(x) |^{\frac{2^*_s}{2^*_\mu-2}}\right)^{\frac{2^*_\mu-2}{2^*_s}} +  \left(\int_{ \{ \underline{k_{v_n}^1} >R_1\}} | \underline{k_{v_n}^1}(x) |^{\frac{2^*_s}{2^*_\mu-2}}\right)^{\frac{2^*_\mu-2}{2^*_s}}\right)\\
&  \hspace{8cm}|  (v_n-\sigma(v_n))^2|_{\frac{22^*_s}{2^*_\mu}}^2. 
\end{align*}
Choose $R_1, R_2>0$ such that, for all $n$, 
\begin{align*}
C_1 S^{-1}\left(\int_{ \{ \underline{k_{v_n}^1} >R_1\}} | \underline{k_{v_n}^1}(x) |^{\frac{2^*_s}{2^*_\mu-2}}\right)^{\frac{2^*_\mu-2}{2^*_s}} < \frac12 \text{ and } C_2  S^{-1/2}\left(\int_{ \{ \underline{k_{v_n}^2} >R_2\}} | \underline{k_{v_n}^2}(x) |^{\frac{2^*_s}{2^*_\mu-1}}\right)^{\frac{2^*_\mu-1}{2^*_s}} < \frac12. 
\end{align*}
Therefore, using the H\"older's inequality in \eqref{si46} with above estimates, we have 
\begin{equation*}
\begin{aligned}
0> 
&  \frac{\|v_n-\sigma(v_n)\|^2}{4} + \varXi(v_n-\sigma(v_n)) -\left(  \frac{C_1R_1}{4}  \left( \int_{ \Om} (v_n-\sigma(v_n))^{\frac{22^*_s}{2^*_\mu}}  ~dx \right)^{\frac{2^*_\mu}{2^*_s}} \right. \\
& \left.\hspace{5cm}+\frac{C_2R_2}{4}  \left( \int_{ \Om} (v_n-\sigma(v_n))^{\frac{2^*_s}{2^*_\mu}}  ~dx \right)^{\frac{2^*_\mu}{2^*_s}}   \|v_n-\sigma(v_n)\|\right)\\
& \geq  \frac{\|(v_n-w_2)^+\|^2}{4} + \varXi((v_n-w_2)^+)\\
& \quad  -\left(  \frac{C_1R_1}{4}  |(v_n-w_2)^+|_{\frac{22^*_s}{2^*_\mu}}^2  +\frac{C_2R_2}{4} |(v_n-w_2)^+|_{\frac{22^*_s}{2^*_\mu}}    \|v_n-\sigma(v_n)\|\right).
\end{aligned}
\end{equation*}
Let $C^*= \max \{ \frac{C_1R_1}{2},  \frac{C_2R_2}{2}   \}$.  Thus 
\begin{equation}\label{si47}
\begin{aligned}
0> 
& \frac{\|(v_n-w_2)^+\|^2}{4} + \varXi((v_n-w_2)^+)\\
&\qquad  - \frac{C^*}{2}\left(   |(v_n-w_2)^+|_{\frac{22^*_s}{2^*_\mu}}^2  +|(v_n-w_2)^+|_{\frac{22^*_s}{2^*_\mu}}    \|(v_n-w_2)^+\|\right).
\end{aligned}
\end{equation}
Let $\nu = \inf \{ \varXi(\phi)\; :\;\phi \in \mathcal{A} \}$ where $\mathcal{A}= \{  \phi \in X_0 \; :\;  \phi\geq 0, |\phi|_{\frac{22^*_s}{2^*_\mu}} =1, \|\phi\| \leq 2 C^*  \}$. 
Clearly, $\mathcal{A}$ is a weakly sequentially closed subset of $X_0$. Using  Fatou's lemma and the fact that Riesz potential  is a bounded linear functional, one can easily prove that $\varXi$ is a weakly lower semicontinuous  on $\mathcal{A}$. Hence $\nu >0$. Indeed, let $z_n$ is a minimizing sequence of $\nu$ such that $z_n \rp z $ weakly in $X_0$ as $n \ra \infty$ then $\varXi(z)\leq \liminf \varXi(z_n)$. Now by the application of the fact that $w_2$ is a strict supersolution to $(\widetilde{P_{\la}})$ we get that $\varXi(z) >0$.  Now notice that using the definition of $\nu$, \eqref{si47} can be rewritten as the following 
	\begin{equation}\label{si95}
	\begin{aligned}
	0>&  \nu + \frac14\left( \left( \|(v_n-w_2)^+\| -  C^* |(v_n-w_2)^+|_{\frac{22^*_s}{2^*_\mu}}   \right)^2 - ((C^*)^2+2C^*)|(v_n-w_2)^+|_{\frac{22^*_s}{2^*_\mu}}^2   \right)\\
	& >  \nu -  \frac14 ((C^*)^2+2C^*)|(v_n-w_2)^+|_{\frac{22^*_s}{2^*_\mu}}^2  
	\end{aligned}
	\end{equation}
As $v_n$ is a sequence such that $v_n \ra u $ in $X_0$. 
It implies that as $n \ra \infty$, $|(v_n-w_2)^+|_{\frac{22^*_s}{2^*_\mu}} \ra 0$. 
So from \eqref{si95}, we get a contradiction to the fact that  $\nu>0$. Hence the proof is complete. 
 \QED
\end{proof}

\begin{Lemma}\label{lemsi13}
	$\La >0$. 
\end{Lemma}

\begin{proof}
We will use  the lower and upper solution method to prove the required result.  From Lemma \ref{lemsi12},  $0$ and $\vartheta$  are the sub and supersolution  respectively to  $(\widetilde{P_\la})$.  We define the closed convex set of $X_0$ as 
\begin{align*}
W= \{ u \in X_0\; :\; 0 \leq u \leq \vartheta  \}. 
\end{align*}
 Employing the  definition of $W$, one can easily prove that 
\begin{align*}
\mc J_W\geq \frac{\|u\|^2}{2}- c_1-c_2
\end{align*}
for  appropriate positive constants $c_1$  and $c_2$. It implies $\mc J_W$ is coercive on $W$.  $\mc J_W$ is weakly lower semi continuous on $W$. Indeed, let $\{ u_n\}  \subset W$ such that $u_n \rp u$ weakly in $X_0$ as $n\ra \infty$. For each $n$, 
\begin{align*}
& \int_{ \Om} G(x,u_n)~dx \leq \int_{ \Om} G(x,u)~dx< +\infty,\\
&  \iint_{\Om\times \Om}\frac{(u_n+\overline{u})^{2^*_\mu}(u_n +\overline{u})^{2^*_\mu}}{|x-y|^{\mu}}~dxdy   \leq \iint_{\Om\times \Om}\frac{(\vartheta+\overline{u})^{2^*_\mu}(\vartheta+\overline{u})^{2^*_\mu}}{|x-y|^{\mu}}~dxdy < +\infty.
\end{align*}
Now we may use the dominated convergence theorem  and the weak lower semicontinuity of the norm to prove that  $\mc J_W$ is weakly lower semi continuous on $W$. Thus, there exists $ u \in X_0$  such that 
\begin{align*}
\inf_{v \in W} \mc J_W(v) = \mc J_W(u).
\end{align*}
 Since $0 \in \pa^- \mc J_W(u)$, $u$ is a weak solution to $(\widetilde{P_{\la}})$. It implies  $\La>0$.  \QED
\end{proof}

\begin{Theorem}\label{thmsi3}
	Let $\la \in (0, \La)$. Then there exists a positive weak solution $u_\la$ to $(\widetilde{P_{\la}})$ belonging to $ X_0$ such that $\mc J(u_\la)<0$ and $u_\la$ is a local minimizer for $\mc J_{K_0}$. 
\end{Theorem}
\begin{proof}
	Let $\la \in (0,\La)$ and $\la^\prime \in (\la, \La )$. Then by Lemma \ref{lemsi12}, $0$ and $u_{\la^\prime}$ are  strict sub and supersolution to $(\widetilde{P_{\la}})$. The existence of $u_{\la^\prime}$ is clear by the definition of $\La$. Consider the convex set $	W= \{ u \in X_0\; :\; 0 \leq u \leq u_{\la^\prime}  \}$. Using the same analysis as in Lemma \ref{lemsi13}, there exists a  $ u_\la \in X_0$  such that$
	\inf_{v \in W} \mc J_W(v) = \mc J_W(u_\la)$. 
	Notice that $0 \in W$ and $\mc J_W(0)<0$, it implies that $\mc J_W(u_\la)<0$. Let $\phi_1=0$ and $\phi_2 = u_{\la^\prime}$ in Theorem \ref{thmsi2} we have $u_{\la}$ is a local minimizer of $\mc J_{K_0}$. \QED
\end{proof}
\begin{Lemma}
	$\La<\infty$. 
\end{Lemma}
\begin{proof}
Assume by contradiction that $\La= +\infty$. Then there exists a sequence $\la_n \ra \infty$ as $n \ra \infty$. Let $u_{\la_n}$ be the corresponding solution to $(\widetilde{P_{\la}})$. Then by Theorem \ref{thmsi3},  $\mc J(u_{\la_n}) <0$  and $u_{\la_n}$ is a local minimizer for $\mc J_{K_0}$.
That is, 
\begin{equation}\label{si51}
\begin{aligned}
& \frac12 \|u_{\la_n}\|^2 + \int_{\Om}G(x,u_{\la_n})~dx - \frac{\la_n}{22^*_\mu} \iint_{\Om\times \Om}\frac{(u_{\la_n}+\overline{u})^{2^*_\mu}(u_{\la_n}+\overline{u})^{2^*_\mu}}{|x-y|^{\mu}}~dxdy<0\\
\text{ and }&  \|u_{\la_n}\|^2 + \int_{\Om}g(x,u_{\la_n})u_{\la_n}~dx - \la_n \iint_{\Om\times \Om}\frac{(u_{\la_n}+\overline{u})^{2^*_\mu}(u_{\la_n}+\overline{u})^{2^*_\mu-1}u_{\la_n}}{|x-y|^{\mu}}~dxdy=0.
\end{aligned}
\end{equation}
With the application of Lemma \ref{lemsi10}(ii) and statements, we have 
\begin{align}\label{si49}
\frac12 \iint_{\Om\times \Om}\frac{(u_{\la_n}+\overline{u})^{2^*_\mu}(u_{\la_n}+\overline{u})^{2^*_\mu-1}u_{\la_n}}{|x-y|^{\mu}}~dxdy < \frac{1}{22^*_\mu} \iint_{\Om\times \Om}\frac{(u_{\la_n}+\overline{u})^{2^*_\mu}(u_{\la_n}+\overline{u})^{2^*_\mu}}{|x-y|^{\mu}}~dxdy.
\end{align}
Using the fact that $ \overline{u} \in L^\infty(\Om)$, for each $ x \in \Om$,  we have 
\begin{align*}
\lim_{t \ra\infty} \frac{\left( \int_{ \Om} \frac{|t+\overline{u}|^{2^*_\mu}}{|x-y|^{\mu }}~dy\right)|t+\overline{u}|^{2^*_\mu} }{\left( \int_{ \Om} \frac{|t+\overline{u}|^{2^*_\mu}}{|x-y|^{\mu }}~dy\right)|t+\overline{u}|^{2^*_\mu-1}t } =1.
\end{align*}
Therefore, it follows that for any small enough $\e>0$, there exists $M_\e >0$ such  that, for all $n$
\begin{equation}
\begin{aligned}\label{si50}
 \frac{1}{22^*_\mu} \iint_{\Om\times \Om}& \frac{(u_{\la_n}+\overline{u})^{2^*_\mu}(u_{\la_n}+\overline{u})^{2^*_\mu}}{|x-y|^{\mu}}~dxdy\\
 &  < \frac{1}{2+\e}\iint_{\Om\times \Om}\frac{(u_{\la_n}+\overline{u})^{2^*_\mu}(u_{\la_n}+\overline{u})^{2^*_\mu-1}u_{\la_n}}{|x-y|^{\mu}}~dxdy +M_\e . 
\end{aligned}
\end{equation}
From \eqref{si49} and \eqref{si50}, we obtain 
\begin{align*}
\iint_{\Om\times \Om}\frac{(u_{\la_n}+\overline{u})^{2^*_\mu}(u_{\la_n}+\overline{u})^{2^*_\mu-1}u_{\la_n}}{|x-y|^{\mu}}~dxdy <\infty \text{ for all } n .
\end{align*}
From \eqref{si51}, we have 
\begin{align*}
\|u_{\la_n}\|^2 <  \la_n \iint_{\Om\times \Om}\frac{(u_{\la_n}+\overline{u})^{2^*_\mu}(u_{\la_n}+\overline{u})^{2^*_\mu-1}u_{\la_n}}{|x-y|^{\mu}}~dxdy. 
\end{align*}
Hence $\{\la_n^{-1/2}u_{\la_n}\}$ is uniformly bounded in $X_0$. Then there exists $z_0 \in X_0$ such that $z_n:=  \la_n^{-1/2}u_{\la_n} \rp z_0$ weakly in $X_0$. Let $0\leq \psi \in C_c^\infty(\Om)$ be a non trivial function. Let $k>0$ such that $\overline{u}>k$ on $\text{supp}(\psi)$. Once again using \eqref{si51}, we deduce that 
\begin{align*}
\sqrt{\la_n} \iint_{\Om\times \Om}\frac{k^{22^*_\mu-1}\psi}{|x-y|^{\mu}}~dxdy & 
\leq \sqrt{\la_n} \iint_{\Om\times \Om}\frac{(u_{\la_n}+\overline{u})^{2^*_\mu}(u_{\la_n}+\overline{u})^{2^*_\mu-1}\psi}{|x-y|^{\mu}}~dxdy\\
& = \ld z_n, \psi\rd + \frac{1}{\sqrt{\la_n}} \int_{\Om}g(x,u_{\la_n})\psi~dx\\
&\leq  \ld z_n, \psi\rd + \frac{1}{\sqrt{\la_n}} \int_{\Om}k^{-q}\psi~dx. 
\end{align*}
Now passing the limit $ n \ra \infty$, we have $ \ld z_0, \psi\rd = \infty$, which is not true. Hence $\La<\infty$. \QED
\end{proof}

\section{Second solution}
In this section we will prove the existence of second solution to $(\widetilde{P_{\la}})$. Let $u_\la$ denotes the first solution to $(\widetilde{P_{\la}})$ obtained in Theorem \ref{thmsi3}. 
\begin{Proposition}\label{propsi4}
	The functional $\mc J_{K_{u_\la}}$ satisfies the $(CPS)_c$ for each $c$ satisfying 
	\begin{align*}
	c < \mc J_{K_{u_\la}}(u_\la) + \frac12 \left(\frac{N-\mu+2s}{2N-\mu}\right) \left( \frac{S_{H,L}^{\frac{2N-\mu}{N-\mu+2s}}}{\la^{\frac{N-2s}{N-\mu+2s}}}  \right). 
	\end{align*}
\end{Proposition}
\begin{proof}
	Let $	c < \mc J_{K_{u_\la}}(u_\la) + \frac12 \left(\frac{N-\mu+2s}{2N-\mu}\right) \left( \frac{S_{H,L}^{\frac{2N-\mu}{N-\mu+2s}}}{\la^{\frac{N-2s}{N-\mu+2s}}}  \right)$. Let $z_n$ be a sequence such that 
\begin{align*}
\mc J_{K_{u_\la}}(z_n) \ra c \text{ and } (1+ \|z_n\|) |||\pa^-\mc J_{K_{u_\la}}(z_n)||| \ra 0 \text{ as } n \ra \infty. 
\end{align*}
It implies there exists $\xi_n \in \pa^-\mc J_{K_{u_\la}}(z_n) $ such that $ \|\xi_n\|= |||\pa^-\mc J_{K_{u_\la}}(z_n)|||$ for every $ n$. From Lemma \ref{lemsi14}, for each $ v \in D(\mc J_{K_{u_\la}}) $ and for each $n$, $g(\cdot, z_n)(v-z_n) \in L^1(\Om)$ and 
\begin{equation}\label{si52}
\begin{aligned}
\ld \xi_n, v-z_n\rd & \leq \ld z_n, v-z_n\rd + \int_{\Om} g(x,z_n)(v-z_n)~dx\\
&  - \la \iint_{\Om\times \Om}\frac{(z_n+\overline{u})^{2^*_\mu}(z_n+\overline{u})^{2^*_\mu-1}(v-z_n)}{|x-y|^{\mu}}~dxdy. 
\end{aligned}
\end{equation} 
Using the fact that $G(\cdot, z_n ) \in L^1(\Om)$  and Lemma \ref{lemsi10}, we obtain that $G(\cdot, 2z_n) \in L^1(\Om)$. So $2z_n \in  D(\mc J_{K_{u_\la}})$, now employing \eqref{si52}, we  get 
\begin{align*}
\ld \xi_n , z_n \rd \leq \|z_n\|^2 + \int_{\Om} g(x,z_n)z_n~dx  - \la \iint_{\Om\times \Om}\frac{(z_n+\overline{u})^{2^*_\mu}(z_n+\overline{u})^{2^*_\mu-1}z_n}{|x-y|^{\mu}}~dxdy. 
\end{align*}
With the help of  Lemma \ref{lemsi10} and \eqref{si50}, for $\e>0$  small enough, 
\begin{equation*}
\begin{aligned}
c+1 &\geq 
\frac12 \|z_n\|^2 + \int_{\Om} G(x,z_n)~dx -\frac{\la}{22^*_\mu} \iint_{\Om\times \Om}\frac{(z_n+\overline{u})^{2^*_\mu}(z_n+\overline{u})^{2^*_\mu}}{|x-y|^{\mu}}~dxdy\\
& \geq \frac12 \|z_n\|^2 + \int_{\Om} G(x,z_n)~dx -\frac{1}{2+\e}\left(\ld \xi_n , z_n \rd - \|z_n\|^2 - \int_{\Om} g(x,z_n)z_n~dx \right)-\la M_\e\\
& \geq \frac12 \|z_n\|^2  -\frac{1}{2+\e}\left(\ld \xi_n , z_n \rd - \|z_n\|^2  \right)-\la M_\e. 
\end{aligned}
\end{equation*} 
It shows that $\{z_n\} $ is a bounded sequence in $X_0$. Hence, up to a subsequence, there exist $z_0 \in X_0$ such that $z_n \rp z_0 $ weakly in $X_0$ as $n \ra \infty$. 
Let $\|z_n- z_0\| \ra a^2 $ and $ \iint_{\Om\times \Om}\frac{(z_n-z_0)^{2^*_\mu}(z_n-z_0)^{2^*_\mu}}{|x-y|^{\mu}}~dxdy \ra b^{22^*_\mu} $ as $n \ra \infty$.  
By the mean value theorem, Brezis-Lieb Lemma (see \cite{ Leib,yang}) and \eqref{si52}, we deduce that 
\begin{align*}
\int_{\Om}G(x,z_0)~dx&  \geq \int_{\Om} G(x,z_n)~dx + \int_{\Om} g(x,z_n)(z_0-z_n)~dx \\
& \geq  \int_{\Om} G(x,z_n)~dx  - \la \iint_{\Om\times \Om}\frac{(z_n+\overline{u})^{2^*_\mu}(z_n+\overline{u})^{2^*_\mu-1}(z_n-z_0)}{|x-y|^{\mu}}~dxdy\\
& \quad -  \ld \xi_n, z_n-z_0\rd + \ld z_n, z_n-z_0\rd \\
& = \int_{\Om} G(x,z_n)~dx  -  \ld \xi_n, z_n-z_0\rd + \ld z_n, z_n-z_0\rd \\
&   - \la \iint_{\Om\times \Om}\frac{\left((z_n-z_0)^{2^*_\mu}(z_n-z_0)^{2^*_\mu} + (z_0+\overline{u})^{2^*_\mu}(z_0+\overline{u})^{2^*_\mu}\right) }{|x-y|^{\mu}}~dxdy\\
& \quad + \la \iint_{\Om\times \Om}\frac{(z_n+\overline{u})^{2^*_\mu}(z_n+\overline{u})^{2^*_\mu-1}(z_0+\overline{u})}{|x-y|^{\mu}}~dxdy.
\end{align*}
Now using the fact that $z_n$ converges to $z_0$ weakly in $X_0$, hence as $n \ra \infty$,  we get 
\begin{align*}
\int_{\Om}G(x,z_0)~dx&  \geq \int_{\Om} G(x,z_0)~dx +a^2 - \la b^{22^*_\mu}. 
\end{align*}
Thus 
\begin{align}\label{si53}
 \la b^{22^*_\mu}\geq a^2 . 
\end{align}
On account of the fact that $u_\la$ is a weak positive solution to $(\widetilde{P_{\la}})$, for each $n$, 
\begin{equation}
\begin{aligned}\label{si54}
0& = \ld u_\la,z_n-u_\la \rd + \int_{\Om} g(x,u_\la)(z_n-u_\la)~dx\\
&   \qquad  +\la  \iint_{\Om\times \Om}\frac{(u_\la+\overline{u})^{2^*_\mu}(u_\la+\overline{u})^{2^*_\mu-1}(z_n-u_\la)}{|x-y|^{\mu}}~dxdy.
\end{aligned}
\end{equation}
In consideration of $G(\cdot, z_n), G(\cdot, 2z_n) \in L^1(\Om)$ and $ u_\la\leq 2z_n-u_\la\leq 2z_n$, we infer that $2z_n-u_\la \in D(\mc J_{K_{u_\la}})$. Testing \eqref{si52} with $2z_n-u_\la$, we obtain that
\begin{equation}\label{si55}
\begin{aligned}
\ld \xi_n, z_n- u_\la \rd & \leq \ld z_n, z_n- u_\la \rd + \int_{\Om} g(x,z_n)(z_n- u_\la)~dx\\
&  - \la \iint_{\Om\times \Om}\frac{(z_n+\overline{u})^{2^*_\mu}(z_n+\overline{u})^{2^*_\mu-1}(z_n- u_\la)}{|x-y|^{\mu}}~dxdy. 
\end{aligned}
\end{equation} 
From Lemma \ref{lemsi10}, \eqref{si54} and \eqref{si55}, we have 
\begin{equation}\label{si56}
\begin{aligned}
&\mc J_{K_{u_\la}}(z_n)- \mc J_{K_{u_\la}}(u_\la)\\
&  = \frac12 \|z_n\|^2 + \int_{\Om}G(x,z_n)~dx - \frac{\la}{22^*_\mu} \iint_{\Om\times \Om}\frac{(z_n+\overline{u})^{2^*_\mu}(z_n+\overline{u})^{2^*_\mu}}{|x-y|^{\mu}}~dxdy \\
& \; -  \frac12 \|u_\la\|^2 - \int_{\Om}G(x,u_\la)~dx + \frac{\la}{22^*_\mu} \iint_{\Om\times \Om}\frac{(u_\la+\overline{u})^{2^*_\mu}(u_\la+\overline{u})^{2^*_\mu}}{|x-y|^{\mu}}~dxdy \\
&\geq  \int_{\Om}\!\!\left(G(x,z_n)- G(x,u_\la) - \frac12 \left(g(x,u_\la) +g(x,z_n)\right)(z_n-u_\la)\right)dx\\
& \; +\frac{\la}{22^*_\mu} \iint_{\Om\times \Om}\frac{(u_\la+\overline{u})^{2^*_\mu}(u_\la+\overline{u})^{2^*_\mu}- (z_n+\overline{u})^{2^*_\mu}(z_n+\overline{u})^{2^*_\mu}}{|x-y|^{\mu}}~dxdy + \frac12 \ld \xi_n, z_n- u_\la \rd \\
&\;  + \frac{\la}{2}  \iint_{\Om\times \Om}\frac{\left( (u_\la+\overline{u})^{2^*_\mu}(u_\la+\overline{u})^{2^*_\mu-1} - (z_n+\overline{u})^{2^*_\mu}(z_n+\overline{u})^{2^*_\mu-1} \right) (z_n-u_\la)}{|x-y|^{\mu}}~dxdy\\
& \geq \frac{\la}{22^*_\mu} \iint_{\Om\times \Om}\frac{(u_\la+\overline{u})^{2^*_\mu}(u_\la+\overline{u})^{2^*_\mu}- (z_n+\overline{u})^{2^*_\mu}(z_n+\overline{u})^{2^*_\mu}}{|x-y|^{\mu}}~dxdy \\
&\;  + \frac{\la}{2}  \iint_{\Om\times \Om}\frac{\left((u_\la+\overline{u})^{2^*_\mu}(u_\la+\overline{u})^{2^*_\mu-1}(z_n-u_\la) - (z_n+\overline{u})^{2^*_\mu}(z_n+\overline{u})^{2^*_\mu-1}( u_\la+\overline{u}) \right)}{|x-y|^{\mu}}~dxdy\\
& \quad  +\frac{\la}{2}  \iint_{\Om\times \Om}\frac{(z_n+\overline{u})^{2^*_\mu}(z_n+\overline{u})^{2^*_\mu}}{|x-y|^{\mu}}~dxdy + \frac12 \ld \xi_n, z_n- u_\la \rd\\
& =: \mc I  + \frac12 \ld \xi_n, z_n- u_\la \rd.
\end{aligned}
\end{equation} 
Using Brezis-Lieb Lemma (See \cite{yang}), we have 
\begin{equation}\label{si57}
\begin{aligned}
\mc I&  =      \la \left(\frac12-  \frac{1}{22^*_\mu} \right) \iint_{\Om\times \Om}\!\!\!\frac{ \left((z_n-z_0)^{2^*_\mu}(z_n-z_0)^{2^*_\mu}\right) + \left( (z_0+ \overline{u})^{2^*_\mu}(z_0+ \overline{u})^{2^*_\mu} \right)}{|x-y|^{\mu}}~dxdy 
\\& 
+  \frac{\la}{2}  \iint_{\Om\times \Om}\frac{\left((u_\la+\overline{u})^{2^*_\mu}(u_\la+\overline{u})^{2^*_\mu-1}(z_n-u_\la)\right) -\left( (z_n+\overline{u})^{2^*_\mu}(z_n+\overline{u})^{2^*_\mu-1}( u_\la+\overline{u})\right)}{|x-y|^{\mu}}~dxdy\\ &\quad  +\frac{\la}{22^*_\mu} \iint_{\Om\times \Om}\frac{(u_\la+\overline{u})^{2^*_\mu}(u_\la+\overline{u})^{2^*_\mu}}{|x-y|^{\mu}}~dxdy    +o(1). 
\end{aligned}
\end{equation}
Observe that by weak convergence of the sequence $\{ z_n \}$, we have 
\begin{equation}\label{si58}
\begin{aligned}
&  \iint_{\Om\times \Om}\!\!\!\frac{(u_\la+\overline{u})^{2^*_\mu}(u_\la+\overline{u})^{2^*_\mu-1}(z_n-z_0)}{|x-y|^{\mu}}~dxdy \ra 0\\
 & \text{ and } \iint_{\Om\times \Om}\frac{\left( (z_n+\overline{u})^{2^*_\mu}(z_n+\overline{u})^{2^*_\mu-1} - (z_0+\overline{u})^{2^*_\mu}(z_0+\overline{u})^{2^*_\mu-1}\right)
 	( u_\la+\overline{u})}{|x-y|^{\mu}}~dxdy \ra 0. 
\end{aligned}
\end{equation}
Taking into account \eqref{si56}, \eqref{si57}, \eqref{si58} and passing the limit as $n \ra \infty$,  we obtain that 
\begin{equation}\label{si59}
\begin{aligned} 
c- \mc J_{K_{u_\la}}(u_\la) & \geq  \frac{\la}{22^*_\mu} \iint_{\Om\times \Om}\frac{(u_\la+\overline{u})^{2^*_\mu}(u_\la+\overline{u})^{2^*_\mu}}{|x-y|^{\mu}}~dxdy +    \la \left(\frac12-  \frac{1}{22^*_\mu} \right) b^{22^*_\mu}
\\& 
\qquad +  \frac{\la}{2}  \iint_{\Om\times \Om}\frac{(u_\la+\overline{u})^{2^*_\mu}(u_\la+\overline{u})^{2^*_\mu-1}(z_0-u_\la)}{|x-y|^{\mu}}~dxdy \\
&  \qquad +  \frac{\la}{2} \iint_{\Om\times \Om}\frac{(z_0+\overline{u})^{2^*_\mu}(z_0+\overline{u})^{2^*_\mu-1}( z_0- u_\la)}{|x-y|^{\mu}}~dxdy\\
&\qquad  -  \frac{\la}{22^*_\mu}  \iint_{\Om\times \Om}\frac{ (z_0+ \overline{u})^{2^*_\mu}(z_0+ \overline{u})^{2^*_\mu}}{|x-y|^{\mu}}~dxdy \\
& \qquad := \mc I_1 (\text{say}) +  \la \left(\frac12-  \frac{1}{22^*_\mu} \right) b^{22^*_\mu}. 
\end{aligned}
\end{equation}
Clearly, we infer
\begin{equation}\label{si60}
\begin{aligned} 
 \mc I_1&  =    \frac{\la}{2}  \iint_{\Om\times \Om}\frac{(u_\la+\overline{u})^{2^*_\mu}\left( (u_\la+\overline{u})^{2^*_\mu-1}+ (z_0+\overline{u})^{2^*_\mu-1}\right) (z_0-u_\la) }{|x-y|^{\mu}}~dxdy \\
 &\qquad   - \frac{\la}{2}  \iint_{\Om\times \Om}\frac{(u_\la+\overline{u})^{2^*_\mu} (z_0+\overline{u})^{2^*_\mu-1}(z_0-u_\la) }{|x-y|^{\mu}}~dxdy
 \\& 
 \qquad +  \frac{\la}{2} \iint_{\Om\times \Om}\frac{(z_0+\overline{u})^{2^*_\mu}\left( (z_0+\overline{u})^{2^*_\mu-1}+ (u_\la+\overline{u})^{2^*_\mu-1}\right) ( z_0- u_\la)}{|x-y|^{\mu}}~dxdy\\
 & \qquad - \frac{\la}{2} \iint_{\Om\times \Om}\frac{(z_0+\overline{u})^{2^*_\mu} (u_\la+\overline{u})^{2^*_\mu-1} ( z_0- u_\la)}{|x-y|^{\mu}}~dxdy\\
 & \qquad + \frac{\la}{22^*_\mu} \iint_{\Om\times \Om}\frac{(u_\la+\overline{u})^{2^*_\mu}\left( (u_\la+\overline{u})^{2^*_\mu}- (z_0+\overline{u})^{2^*_\mu}\right)}{|x-y|^{\mu}}~dxdy
 \\& 
  \qquad  + \frac{\la}{22^*_\mu} \iint_{\Om\times \Om}\frac{(z_0+\overline{u})^{2^*_\mu}\left( (u_\la+\overline{u})^{2^*_\mu}- (z_0+\overline{u})^{2^*_\mu}\right)}{|x-y|^{\mu}}~dxdy.
\end{aligned}
\end{equation}
Since 
\begin{align*}
(u_\la+\overline{u})^{2^*_\mu}- (z_0+\overline{u})^{2^*_\mu} & = -2^*_\mu \int_{u_\la}^{z_0} (t+\overline{u})^{2^*_\mu-1}~dt \\
& \quad 
\geq -2^*_\mu \left( \frac{(u_\la+\overline{u})^{2^*_\mu-1}+ (z_0+\overline{u})^{2^*_\mu-1}}{2} \right) (z_0-u_\la). 
\end{align*}
It implies 
\begin{equation}\label{si61}
\begin{aligned} 
\frac{\la}{22^*_\mu}& \iint_{\Om\times \Om}\frac{(u_\la+\overline{u})^{2^*_\mu}\left( (u_\la+\overline{u})^{2^*_\mu}- (z_0+\overline{u})^{2^*_\mu}\right)}{|x-y|^{\mu}}~dxdy\\
&\geq  -  \frac{\la}{4}  \iint_{\Om\times \Om}\frac{(u_\la+\overline{u})^{2^*_\mu}\left( (u_\la+\overline{u})^{2^*_\mu-1}+ (z_0+\overline{u})^{2^*_\mu-1}\right) (z_0-u_\la) }{|x-y|^{\mu}}~dxdy.  
\end{aligned}
\end{equation}
\begin{equation}\label{si62}
\begin{aligned} 
\text{Similarly, }& \frac{\la}{22^*_\mu}\iint_{\Om\times \Om}\frac{(z_0+\overline{u})^{2^*_\mu}\left( (u_\la+\overline{u})^{2^*_\mu}- (z_0+\overline{u})^{2^*_\mu}\right)}{|x-y|^{\mu}}~dxdy\\
&\geq  -  \frac{\la}{4}  \iint_{\Om\times \Om}\frac{(z_0+\overline{u})^{2^*_\mu}\left( (u_\la+\overline{u})^{2^*_\mu-1}+ (z_0+\overline{u})^{2^*_\mu-1}\right) (z_0-u_\la) }{|x-y|^{\mu}}~dxdy.  
\end{aligned}
\end{equation}
From \eqref{si60}, \eqref{si61} and \eqref{si62}, we deduce that 
\begin{equation}\label{si63}
\begin{aligned} 
\mc I_1&  =    \frac{\la}{4}  \iint_{\Om\times \Om}\frac{(u_\la+\overline{u})^{2^*_\mu}\left( (u_\la+\overline{u})^{2^*_\mu-1}- (z_0+\overline{u})^{2^*_\mu-1}\right) (z_0-u_\la) }{|x-y|^{\mu}}~dxdy \\
&\qquad +  \frac{\la}{4} \iint_{\Om\times \Om}\frac{(z_0+\overline{u})^{2^*_\mu}\left( (z_0+\overline{u})^{2^*_\mu-1}- (u_\la+\overline{u})^{2^*_\mu-1}\right) ( z_0- u_\la)}{|x-y|^{\mu}}~dxdy\\
& =  \frac{\la}{4}  \iint_{\Om\times \Om}\frac{\left((z_0+\overline{u})^{2^*_\mu} -  (u_\la+\overline{u})^{2^*_\mu}\right) \left( (z_0+\overline{u})^{2^*_\mu-1}- (u_\la+\overline{u})^{2^*_\mu-1}\right) ( z_0- u_\la) }{|x-y|^{\mu}}~dxdy \\
& \geq 0. 
\end{aligned}
\end{equation}
Hence from \eqref{si59} and \eqref{si63}, we obtain 
\begin{equation}\label{si64}
\begin{aligned} 
c- \mc J_{K_{u_\la}}(u_\la) & \geq   \la \left(\frac12-  \frac{1}{22^*_\mu} \right) b^{22^*_\mu}. 
\end{aligned}
\end{equation}
Using definition of $S_{H,L}$ and \eqref{si53}, we have $ \la b^{22^*_\mu}\geq a^2$ and $a^2 \geq S_{H,L}b^2$, that is
\begin{align}\label{si65}
b\geq \left( \frac{S_{H,L}}{\la}\right)^{\frac{N-2s}{2(N-\mu+2s)}}. 
\end{align}
Using \eqref{si64} and \eqref{si65}, we get 
\begin{align*}
c- \mc J_{K_{u_\la}}(u_\la) & \geq   \la \left(\frac12-  \frac{1}{22^*_\mu} \right) \left( \frac{S_{H,L}}{\la}\right)^{\frac{2N-\mu}{N-\mu+2s}} =  \frac12 \left(\frac{N-\mu+2s}{2N-\mu}\right) \left( \frac{S_{H,L}^{\frac{2N-\mu}{N-\mu+2s}}}{\la^{\frac{N-2s}{N-\mu+2s}}}  \right). 
\end{align*}
It contradicts the fact that $	c < \mc J_{K_{u_\la}}(u_\la) + \frac12 \left(\frac{N-\mu+2s}{2N-\mu}\right) \left( \frac{S_{H,L}^{\frac{2N-\mu}{N-\mu+2s}}}{\la^{\frac{N-2s}{N-\mu+2s}}}  \right)$. Hence $a=0$. \QED
\end{proof}
Now consider the family of minimizers $\{U_\e\}_{\e>0}$ of $S$ defined as 
\begin{align*}
U_\e = \e^{-\frac{(N-2s)}{2}} S^{\frac{(N-\mu)(2s-N)}{4(N-\mu+2s)}}(C(N,\mu))^{\frac{2s-N}{2(N-\mu+2s)}} u^*(x/\e)
\end{align*}
where $u^*(x)= \overline{u}(x/S^{1/2s}),\; \overline{u} (x)= \frac{\tilde{u}(x)}{|\tilde{u}|_{2^*_s}}$ and $\tilde{u}(x)= a(b^2+|x|^2)^{\frac{-(N-2s)}{2}}$ with $ \al \in \mathbb{R}\setminus\{0\}$ and  $\ba >0$ are fixed constants. Then  from Lemma \ref{lulem13}, for $\e>0, ~U_\e$ satisfies 
\begin{align*}
(-\De)^s u = (|x|^{-\mu}* |u|^{2^*_\mu})|u|^{2^*_\mu-2}u \text{ in } \R. 
\end{align*}

Let $\varrho >0$ such that $B_{4\varrho} \subset \Om$. Now  define $\eta\in C_c^{\infty}(\mathbb{R}^N)$ such that $0\leq \eta\leq 1$ in $\mathbb{R}^N$, $\eta\equiv1$ in $B_\varrho(0)$ and $\eta\equiv 0 $ in $\mathbb{R}^N \setminus B_{2\varrho}(0)$. For each $\e>0 $ and  $x \in \R$, we define $\varPsi_\e = \eta(x) U_\e(x) $. 
\begin{Proposition}\label{propsi3}
	Let $N> 2s,\; 0<\mu<N$ then the following holds:
	\begin{enumerate}
		\item [(i)] $\|\varPsi_\e\|^2 \leq  S_{H,L}^{\frac{2N-\mu}{N-\mu+2s}}+O(\e^{N-2s})$.
		\item [(ii)] $\|\varPsi_\e\|_{NL}^{2.2^*_{\mu}}\leq S_{H,L}^{\frac{2N-\mu}{N-\mu+2s}}+O(\e^N)$.
		\item [(iii)] $\|\varPsi_\e\|_{NL}^{2.2^*_{\mu}}\geq S_{H,L}^{\frac{2N-\mu}{N-\mu+2s}}-O(\e^N)$.
	\end{enumerate}
\end{Proposition}
\begin{proof}
	Using the definition of $\varPsi_\e$ and \cite[Proposition 1]{sevadei1} part $(i)$ follows.  For $(ii)$ and $(iii)$ see \cite[Proposition 2.8]{GMS1}.\QED
\end{proof}
\begin{Lemma}\label{lemsi5}
	\cite{ds} The following holds: 
	\begin{enumerate}
		\item[(i)]If  $\mu<\min\{4s,N\}$ then  for all $ \Theta <1 \nonumber$, 
		\begin{align*}
		\|u_\la+t \varPsi_\e\|_{NL}^{2.2^*_{\mu}}	& 
		\geq \|u_\la\|_{NL}^{2.2^*_{\mu}}+\|\varPsi_\e\|_{NL}^{2.2^*_{\mu}} +\widehat{C} t^{2.2^*_{\mu}-1} \int_{\Om}\int_{ \Om}\frac{(\varPsi_\e(x))^{2^*_{\mu}}(\varPsi_\e(y))^{2^*_{\mu}-1}u_\la(y)}{|x-y|^{\mu}}~dxdy\\& \quad + 2.2^*_{\mu} t \int_{\Om}\int_{ \Om}\frac{(u_\la(x))^{2^*_{\mu}}(u_\la(y))^{2^*_{\mu}-1}\varPsi_\e(y)}{|x-y|^{\mu}}~dxdy- O(\e^{(\frac{2N-\mu}{4}) \Theta}) .
		\end{align*}
		\item [(ii)]There exists a $R_0>0$ such that 
		$
		\int_{\Om}\int_{\Om}\frac{(\varPsi_\e(x))^{2^{*}_{\mu}}(\varPsi_\e(y))^{2^{*}_{\mu}-1}u_\la(y)}{|x-y|^{\mu}}~dxdy\geq \widehat{C}R_0\e^{\frac{N-2s}{2}}$. 
	\end{enumerate}
	
\end{Lemma}

\begin{Lemma}\label{lemsi6}
	$\sup\{  \mc J_{K_{u_\la}} (u_\la+ t \varPsi_\e): t\geq 0 \}< \mc J_{K_{u_\la}}(u_\la)+ \frac12\left( \frac{N-\mu +2s}{2N-\mu}\right) \frac{S_{H,L}^{\frac{2N-\mu}{N-\mu+2s}}}{\la^{\frac{N-2s}{N-\mu+2s}}}$ for any sufficiently small $\e>0$. 
\end{Lemma} 
\begin{proof}
	Employing  the fact that $u_\la$ is a weak solution to $(P_\la)$  and using Lemma \ref{lemsi5}, for all $\Theta<1$, we have 
	\begin{align*}
	\mc J_{K_{u_\la}}(u_\la+ t\varPsi_\e) - \mc J_{K_{u_\la}}(u_\la)&  \leq  \frac12\|t\varPsi_\e\|^2- \frac{\la}{22^*_\mu} \|t \varPsi_\e\|_{NL}^{2.2^*_{\mu}}+ O(\e^{(\frac{2N-\mu}{4}) \Theta}) \\&\quad +  \int_{ \Om}(G(u_\la+t\varPsi_\e)-G(x,u_\la) - g(x,u_\la)t\varPsi_\e)~dx  
\\&\quad	- \frac{\la\widehat{C} t^{2.2^*_{\mu}-1}}{22^*_\mu} \int_{\Om}\int_{ \Om}\frac{(\varPsi_\e(x))^{2^*_{\mu}}(\varPsi_\e(y))^{2^*_{\mu}-1}u_\la(y)}{|x-y|^{\mu}}~dxdy . 
	\end{align*}
	From Proposition \ref{propsi3} and Lemma \ref{lemsi5}, we deduce  that  
	\begin{equation}\label{si27}
	\begin{aligned}
	\mc J_{K_{u_\la}}(u_\la+ t\varPsi_\e) - \mc J_{K_{u_\la}}(u_\la)&  \leq  \frac{t^2}{2}(S_{H,L}^{\frac{2N-\mu}{N-\mu+2s}}+O(\e^{N-2s}))- \frac{\la t^{22^*_\mu}}{22^*_\mu} (S_{H,L}^{\frac{2N-\mu}{N-\mu+2s}}-O(\e^N)) 
	\\&\quad +  \int_{ \Om}(G(u_\la+t\varPsi_\e)-G(x,u_\la) - g(x,u_\la)t\varPsi_\e)~dx 
	\\&\quad - \frac{\widehat{C} t^{2.2^*_{\mu}-1}}{22^*_\mu} \widehat{C}R_0\e^{\frac{N-2s}{2}}   + O(\e^{(\frac{2N-\mu}{4}) \Theta}). 
	\end{aligned}
	\end{equation}
	Observe that for any fix $1<\rho < \min \{2, \frac{2}{n-2s}  \}$, there exists $R_1 >0$ such that 
	\begin{align*}
	\int_{ \Om}|\varPsi_\e|^\rho ~dx \leq R_1 \e^{(n-2s)\rho/2}. 
	\end{align*}
	Moreover, there exists $ R_2 >0 $ such that,  for all  $x \in \Om, r>m\text{ and } s \geq 0$, 
	\begin{align*}
	G(x,r+s)- G(x,s)-g(x,r)s = \int_r^{r+s} (\tau^{-q}- r^{-q})~d\tau \leq R_2 s^{\rho} . 
	\end{align*}
	Using  last inequality and \eqref{si27} with $\Theta= \frac{2}{2^*_\mu}$, we obtain 
	\begin{equation*}\label{si28}
	\begin{aligned}
	\mc J_{K_{u_\la}}(u_\la+ t\varPsi_\e) - \mc J_{K_{u_\la}}(u_\la)&  \leq  \frac{t^2}{2}(S_{H,L}^{\frac{2N-\mu}{N-\mu+2s}}+O(\e^{N-2s}))- \frac{t^{22^*_\mu}}{22^*_\mu} (S_{H,L}^{\frac{2N-\mu}{N-\mu+2s}}-O(\e^N)) \\&\quad  - \frac{\widehat{C} t^{2.2^*_{\mu}-1}}{22^*_\mu} \widehat{C}R_0\e^{\frac{N-2s}{2}}  + R_1R_2t^\rho \e^{(n-2s)\rho/2} + o(\e^{\frac{N-2s}{2}})\\
	& := K(t). 
	\end{aligned}
	\end{equation*}
	Clearly, one can check that $K(t)\ra -\infty, K(t)>0$ as $t \ra 0^+$ and there exists $t_\e >0$ such that $K^\prime(t_\e)=0$. Furthermore, there exist  positive constants $T_1$ and $T_2$ such that $T_1\leq t_\e \leq T_2$ (for details see \cite{ds}). Hence,
	\begin{align*}
	K(t)& \leq \frac{t_\e^2}{2}(S_{H,L}^{\frac{2N-\mu}{N-\mu+2s}}+O(\e^{N-2s}))- \frac{t_\e^{22^*_\mu}}{22^*_\mu} (S_{H,L}^{\frac{2N-\mu}{N-\mu+2s}}-O(\e^N))  - \frac{\widehat{C} T_1^{2.2^*_{\mu}-1}}{22^*_\mu} \widehat{C}R_0\e^{\frac{N-2s}{2}}\\&\quad  + R_1R_2T_2^\rho \e^{(n-2s)\rho/2} + o(\e^{\frac{N-2s}{2}})\\
	& \leq \sup_{t\geq 0} K_1(t) - \frac{\widehat{C} T_1^{2.2^*_{\mu}-1}}{22^*_\mu} \widehat{C}R_0\e^{\frac{N-2s}{2}} + R_1R_2T_2^\rho \e^{(n-2s)\rho/2} + o(\e^{\frac{N-2s}{2}})
	\end{align*}
	where $K_1(t)=\frac{t^2}{2}(S_{H,L}^{\frac{2N-\mu}{N-\mu+2s}}+O(\e^{N-2s}))- \frac{t^{22^*_\mu}}{22^*_\mu} (S_{H,L}^{\frac{2N-\mu}{N-\mu+2s}}-O(\e^N))$. By trivial computations, we get
	\begin{align*}
	\mc J_{K_{u_\la}}(u_\la+ t\varPsi_\e) - \mc J_{K_{u_\la}}(u_\la)&  \leq  \frac12\left( \frac{N-\mu +2s}{2N-\mu}\right) \frac{S_{H,L}^{\frac{2N-\mu}{N-\mu+2s}}}{\la^{\frac{N-2s}{N-\mu+2s}}} +O(\e^{\frac{N-2s}{2}}) -C\e^{\frac{N-2s}{2}} + o(\e^{\frac{N-2s}{2}})
	\end{align*}
	for an appropriate constant $C>0$.  Thus, for $\e$ sufficiently small,  
	\begin{align*}
	\mc J_{K_{u_\la}}(u_\la+ t\varPsi_\e) - \mc J_{K_{u_\la}}(u_\la)&  <   \frac12\left( \frac{N-\mu +2s}{2N-\mu}\right) \frac{S_{H,L}^{\frac{2N-\mu}{N-\mu+2s}}}{\la^{\frac{N-2s}{N-\mu+2s}}} . 
	\end{align*}
	Hence the proof follows. \QED
\end{proof}
\begin{Proposition}\label{propsi5}
	For each $\la \in (0, \La)$ there exist a second positive solution to $(\widetilde{P_{\la}})$. 
\end{Proposition}
\begin{proof}
From Theorem \ref{thmsi3}, $u_\la$ is a local  minimizer of $	\mc J_{K_{u_\la}}$. It implies there exist $\varsigma >0$ such that $ 	\mc J_{K_{u_\la}}(z)\geq 	\mc J_{K_{u_\la}}(u_\la)$ for every $z \in K_{u_\la}$ with $\|z-u_\la\|\leq \varsigma$. Let $\varPsi= \varPsi_\e$ for $\e$ obtained in Lemma \ref{lemsi6}.  Since $	\mc J_{K_{u_\la}}(u_\la+t\varPsi)\ra -\infty$ as $t \ra \infty$, so choose $t\geq \varsigma/\|\varPsi\| $ such that $ 	\mc J_{K_{u_\la}}(u_\la+t\varPsi)\leq 	\mc J_{K_{u_\la}}(u_\la)$. Define 
\begin{align*}
& \varSigma= \{ \phi \in C([0,1], D(	\mc J_{K_{u_\la}}))\; : \; \phi(0) = u_\la, \phi(1)= u_\la+t\varPsi  \},\\
& A = \{ z \in D(	\mc J_{K_{u_\la}})\;:\; \|z-u_\la\|= \al \} \text{ and } c= \inf_{\phi \in \varSigma}\sup_{r\in [0,1]} 	\mc J_{K_{u_\la}}(\phi(r)). 
\end{align*}
With the help of Proposition \ref{propsi4} and Lemma \ref{lemsi6}, $	\mc J_{K_{u_\la}}$ satisfies $(\text{CPS})_c$ condition. If $c= 	\mc J_{K_{u_\la}}(u_\la)= \inf	\mc J_{K_{u_\la}}(A)$ then $u_\la  \not \in A,\; u_\la +t\varPsi \not \in A, 
\; \inf	\mc J_{K_{u_\la}}(A) \geq 	\mc J_{K_{u_\la}}(u_\la) \geq  \mc J_{K_{u_\la}}(u_\la+t\varPsi)$, and for every $\phi \in \varSigma$, there exist $r \in [0,1]$ such that $\|\phi(r)-u_\la\|= \varsigma$. Thus by Theorem \ref{thmsi4}, we get there exists $z_\la \in D(\mc J_{K_{u_\la}})$ such that $z_\la \not = u_\la,\; \mc J_{K_{u_\la}}(z_\la) =c$ and $ 0 \in \pa^-\mc J_{K_{u_\la}}(z_\la)$. Using Proposition \ref{propsi2}, we obtain that $z_\la$ is positive weak solution to $(\widetilde{P_{\la}})$. \QED
\end{proof}

{\bf Proof of Theorem \ref{thmsi1}}: It follows from Theorem \ref{thmsi3}, Proposition \ref{propsi5} and Lemma \ref{lemsi2}. \QED

\end{document}